\documentclass[leqno,10pt]{article}
\usepackage{theorem}
\usepackage{amssymb}
\usepackage{latexsym}

\theorembodyfont{\itshape}
\newtheorem{Theorem}{Theorem}[section]
\newtheorem{Proposition}{Proposition}[section]

\newtheorem{Lemma}{Lemma}[section]
\newtheorem{Claim}{Claim}[section]
\newtheorem{con}{Conjecture}

\newtheorem{prob}{Problem}
\newtheorem{Corollary}{Corollary}[section]
\newtheorem{thm}{Theorem}
\newtheorem{s-Theorem}{Theorem}[subsection]
\newtheorem{s-Proposition}{Proposition}[subsection]
\newtheorem{s-Conjecture}{Conjecture}[subsection]
\newtheorem{s-Lemma}{Lemma}[subsection]
\newtheorem{s-Claim}{Claim}[subsection]
\newtheorem{s-Problem}{Problem}[subsection]
\newtheorem{s-Corollary}{Corollary}[subsection]

\theorembodyfont{\rmfamily}
\newtheorem{Definition}{Definition}[section]
\newtheorem{Notation}{Notation}[section]
\newtheorem{Remark}{Remark}[section]

\newtheorem{s-Definition}{Definition}[subsection]
\newtheorem{s-Notation}{Notation}[subsection]
\newtheorem{s-Remark}{Remark}[subsection]
\newtheorem{s-Example}{Example}[subsection]

\DeclareRobustCommand{\qed}{%
  \ifmmode % if math mode, assume display: omit penalty etc.
  \else \leavevmode\unskip\penalty9999 \hbox{}\nobreak\hfill
  \fi
  \quad\hbox{\qedsymbol}}
\newcommand{\openbox}{\leavevmode
  \hbox to.77778em{%
  \hfil\vrule
  \vbox to.675em{\hrule width.6em\vfil\hrule}%
  \vrule\hfil}}
\newcommand{\qedsymbol}{\openbox}
\newenvironment{proof}[1][\proofname]{\par
  \vskip \parsep%
  \normalfont
  {\itshape
  #1{.}}\quad\mdseries \ignorespaces
}{%
  \qed%\endtrivlist
  {\vskip \belowdisplayskip}
}
\newcommand{\proofname}{Proof}
\begin{document}
\title{{\bf Effective non-vanishing of global sections of multiple adjoint bundles for polarized \textrm{\boldmath $4$}-folds}
\thanks{This research was partially supported by the Grant-in-Aid for Scientific Research (C)
(No.20540045), Japan Society for the Promotion of Science, Japan.}}
\author{YOSHIAKI FUKUMA}
\date{}
\maketitle
\begin{abstract}
Let $X$ be a smooth complex projective variety of dimension $4$ and let $L$ be an ample line bundle on $X$.
In this paper, we study a natural number $m$ such that $h^{0}(m(K_{X}+L))>0$  for any polarized $4$-folds $(X,L)$ with $\kappa(K_{X}+L)\geq 0$.
\end{abstract}

\section{Introduction}
Let $X$ be a smooth projective variety of dimension $n$ defined over the field of complex numbers and let $L$ be an ample (resp. nef and big) line bundle on $X$.
Then the pair $(X,L)$ is called a polarized (resp. quasi-polarized) manifold.
\par
Then there are the following conjectures.

\begin{con}\label{Conjecture}
\begin{itemize}
\item [\rm (i)] {\rm \textbf{(Ionescu \cite[Open problems, P.321]{LPS93})}}
Let $(X,L)$ be a quasi-polarized manifold of dimension $n$.
Assume that $K_{X}+L$ is nef.
Then $h^{0}(K_{X}+L)>0$. 
\item [\rm (ii)] {\rm \textbf{(Ambro \cite[Section 4]{Ambro}, Kawamata \cite[Conjecture 2.1]{Kawamata})}}
Let $X$ be a complex normal variety, 
$B$ an effective $\mathbb{R}$-divisor on $X$ 
such that the pair $(X,B)$ is KLT, and $D$ a Cartier divisor on $X$.
Assume that $D$ is nef, and that $D-(K_{X}+B)$ is nef and big.
Then $h^{0}(D)>0$.
\end{itemize}
\end{con}

These conjectures have been studied by several authors (see \cite{Kawamata}, \cite{1.5}, \cite{Fukuma07}, \cite{Fukuma10}, \cite{Broustet09}, \cite{BrHo09}, \cite{Horing09}).
In particular it is known that Conjecture \ref{Conjecture} (i) is true if $\dim X\leq 3$, and Conjecture \ref{Conjecture} (ii) is true if $\dim X\leq 2$.
\par
Here we note that if $K_{X}+L$ is nef, then by \cite{Shokurov86} there exists a positive integer $m$ such that $h^{0}(m(K_{X}+L))>0$, that is, $\kappa(K_{X}+L)\geq 0$.
So, as a generalization of Conjecture \ref{Conjecture} (i),
it is natural and interesting to study the following problem, which was proposed in \cite[Problem 3.2]{Fukuma07}:

\begin{prob}\label{Problem2.8}
For any fixed positive integer $n$, determine the smallest positive integer 
$p$, which depends only on $n$, such that the following {\rm ($*$)} is satisfied:
\begin{itemize}
\item [\rm ($*$)]
$h^{0}(p(K_{X}+L))>0$ for any polarized manifold $(X,L)$ of dimension $n$ 
with $\kappa(K_{X}+L)\geq 0$.
\end{itemize}
\end{prob}

The aim of this paper is to study Problem \ref{Problem2.8}.
It is known that $p=1$ if $X$ is a curve or surface 
(see \cite[Theorem 2.8]{Fukuma07}).
For the case of $n=3$, in \cite[Theorems 5.1 and 5.2]{Fukuma10}, we proved $p\leq 2$.
Concretely, in \cite[Theorem 5.4 (2)]{Fukuma08-3} or \cite[Theorem 5.2]{Fukuma10}, we proved that if $\kappa(K_{X}+L)=3$, then $h^{0}(2(K_{X}+L))\geq 3$.
Moreover in \cite[Theorem 5.1]{Fukuma10}, we proved that $h^{0}(K_{X}+L)>0$ if $0\leq\kappa(K_{X}+L)\leq 2$.
Moreover by using a result of H\"oring (\cite[1.5 Theorem]{Horing09}) and the adjunction theory, we can get $p=1$ if $\dim X=3$.
Therefore Problem \ref{Problem2.8} was completly solved for the case of $\dim X=3$.
\par
So, as the next step, in this paper, we will treat the case of $\dim X=4$.
The main result of this paper is the following.

\begin{thm}\label{MainTheorem1}
Let $(X,L)$ be a polarized manifold of dimension $4$.
Assume that $\kappa(K_{X}+L)\geq 0$ holds.
\begin{itemize}
\item [\rm (a)]
If $0\leq \kappa(K_{X}+L)\leq 2$, then $h^{0}(2m(K_{X}+L))>0$ holds for every integer $m$ with $m\geq 1$. In particular $h^{0}(2(K_{X}+L))>0$.
\item [\rm (b)]
If $\kappa(K_{X}+L)=3$, then $h^{0}(2m(K_{X}+L))>0$ holds for every integer $m$ with $m\geq 2$. In particular $h^{0}(4(K_{X}+L))>0$.
\item [\rm (c)]
If $\kappa(K_{X}+L)=4$, then $h^{0}(2m(K_{X}+L))>0$ holds for every integer $m$ with $m\geq 3$. In particular $h^{0}(6(K_{X}+L))>0$.
\item [\rm (d)]
If $\kappa(X)\geq 0$, then 
$$h^{0}(2m(K_{X}+L))\geq \frac{(m-1)(m-2)(m^{2}+3m+6)}{12}+1$$
for every integer $m\geq 2$.
In particular $h^{0}(4(K_{X}+L))>0$.
\end{itemize}
\end{thm}

By using the adjunction theory of Beltrametti and Sommese, Theorem \ref{MainTheorem1} is obtained from the following result, which will be proved in this paper (see Theorems \ref{Theorem4.1}, \ref{Theorem4.1-big}, \ref{T-TH1}, \ref{Theorem4.1.5} and \ref{Theorem4.2}).

\begin{thm}\label{MainTheorem2}
Let $(X,L)$ be a polarized variety of dimension $4$ such that $X$ is a normal Gorenstein projective variety with only isolated terminal singularities.
Assume that $K_{X}+L$ is nef.
\begin{itemize}
\item [\rm (a)]
If $0\leq \kappa(K_{X}+L)\leq 2$, then $h^{0}(m(K_{X}+L))>0$ for every integer $m$ with $m\geq 1$.
\item [\rm (b)]
If $\kappa(K_{X}+L)=3$, then $h^{0}(m(K_{X}+L))>0$ for every integer $m$ with $m\geq 2$.
\item [\rm (c)]
If $\kappa(K_{X}+L)=4$, then $h^{0}(m(K_{X}+L))>0$ for every integer $m$ with $m\geq 3$.
\item [\rm (d)]
If $\kappa(X)\geq 0$, then 
$$h^{0}(m(K_{X}+L))\geq \frac{(m-1)(m-2)(m^{2}+3m+6)}{12}+1$$
for every integer $m\geq 2$.
In particular $h^{0}(2(K_{X}+L))>0$.
\end{itemize}
\end{thm}

Here we note that recently Arakawa \cite[Theorem 1.5]{Arakawa10} proved that $h^{0}(m(K_{X}+L))>0$ for any polarized $n$-folds $(X,L)$ such that $K_{X}+L$ is nef, and for every integer $m$ with $m\geq (n(n+1)/2)+2$.
\\
\par
In this paper, we shall study mainly a smooth projective variety $X$ 
over the field of complex numbers $\mathbb{C}$.
We will employ the customary notation in algebraic geometry.

\section{Preliminaries}

\begin{Definition}\label{DF3}
(1) Let $X$ (resp. $Y$) be an $n$-dimensional projective manifold, and let $\mathcal{L}$ (resp. $\mathcal{A}$) be an ample line bundle on $X$ (resp. $Y$).
Then $(X,\mathcal{L})$ is called a  {\it simple blowing up of $(Y,\mathcal{A})$} if there exists a birational morphism $\pi: X\to Y$ such that $\pi$ is a blowing up at a point of $Y$ and $\mathcal{L}=\pi^{*}(\mathcal{A})-E$, where $E$ is the exceptional divisor.
\\
(2) Let $X$ (resp. $M$) be an $n$-dimensional projective manifold, and let $\mathcal{L}$ (resp. $\mathcal{A}$) be an ample line bundle on $X$ (resp. $M$).
Then we say that $(M,\mathcal{A})$ is a {\it reduction of $(X,\mathcal{L})$} if $(X,\mathcal{L})$ is obtained by a composite of simple blowing ups of $(M,\mathcal{A})$, and $(M,\mathcal{A})$ is not obtained by a simple blowing up of any polarized manifold.
The morphism $\mu:X\to M$ is called the {\it reduction map}.
\end{Definition}

\begin{Definition}\label{DF4}
Let $(X,\mathcal{L})$ be a polarized manifold of dimension $n$.
We say that $(X,\mathcal{L})$ is a {\it scroll} (resp. {\it quadric fibration}, {\it Del Pezzo fibration}) {\it over a normal projective variety $Y$ with $\dim Y=m$} if there exists a surjective morphism with connected fibers $f:X\to Y$ such that $K_{X}+(n-m+1)\mathcal{L}=f^{*}\mathcal{A}$ (resp. $K_{X}+(n-m)\mathcal{L}=f^{*}\mathcal{A}$, $K_{X}+(n-m-1)\mathcal{L}=f^{*}\mathcal{A}$) for some ample line bundle $\mathcal{A}$ on $Y$.
\end{Definition}

\begin{Remark}\label{RM3}
If $(X,\mathcal{L})$ is a scroll over a smooth curve $C$ 
(resp. a smooth projective surface $S$) with $\dim X=n\geq 3$,
then by \cite[(3.2.1) Theorem]{BeSoWi} and \cite[Proposition 3.2.1]{BeSo-Book}
there exists an ample vector bundle $\mathcal{E}$ of rank $n$ (resp. $n-1$)
on $C$ (resp. $S$) such that 
$(X,\mathcal{L})\cong (\mathbb{P}_{C}(\mathcal{E}),H(\mathcal{E}))$ 
(resp. $(\mathbb{P}_{S}(\mathcal{E}),H(\mathcal{E}))$).
\end{Remark}

Here we give the definition of the $i$th sectional geometric genus of multi-prepolarized varieties.

\begin{Notation}\label{N0}
Let $X$ be a projective variety of dimension $n$, let $i$ be an integer with $0\leq i\leq n-1$, let $L_{1},\dots , L_{n-i}$ be line bundles on $X$ and let $\mathcal{F}$ be a coherent sheaf on $X$.
Then $\chi(L_{1}^{t_{1}}\otimes\cdots\otimes L_{n-i}^{t_{n-i}}\otimes \mathcal{F})$ is a polynomial in $t_{1}, \dots ,t_{n-i}$ of total degree at most $n$.
So we can write $\chi(L_{1}^{t_{1}}\otimes\cdots
\otimes L_{n-i}^{t_{n-i}}\otimes \mathcal{F})$ uniquely as follows.
\begin{eqnarray*}
&&\chi(L_{1}^{t_{1}}\otimes\cdots\otimes L_{n-i}^{t_{n-i}}\otimes \mathcal{F}) \\
&&=\sum_{p=0}^{n}\sum
_{\stackrel{p_{1}\geq 0,\dots , p_{n-i}\geq 0}
{p_{1}+\cdots +p_{n-i}=p}}
\chi_{p_{1},\dots , p_{n-i}}(L_{1},\dots ,L_{n-i};\mathcal{F})
{t_{1}+p_{1}-1\choose p_{1}}\dots{t_{n-i}+p_{n-i}-1\choose p_{n-i}}.
\end{eqnarray*}
\end{Notation}

\begin{Definition}\label{B2}
Let $X$ be a projective variety of dimension $n$, let $i$ be an integer with $0\leq i\leq n-1$, let $L_{1},\dots , L_{n-i}$ be line bundles on $X$ and let $\mathcal{F}$ be a coherent sheaf on $X$.
\\
(1) For every $i\in \mathbb{Z}$ with $0\leq i\leq n$, we set

\[
\chi_{i}^{H}(X,L_{1},\dots , L_{n-i};\mathcal{F}):=
\left\{
\begin{array}{ll}
\chi_{\underbrace{1, \dots , 1}_{n-i}}(L_{1},\dots , L_{n-i};\mathcal{F}) 
& \ \ \mbox{if $0\leq i\leq n-1$,} \\
\chi(\mathcal{F}) & \ \ \mbox{if $i=n$.}
\end{array}\right. \] 
\\
(2)  For every $i\in \mathbb{Z}$ with $0\leq i\leq n$,
the {\it $i$th sectional geometric genus $g_{i}(X,L_{1},\dots , L_{n-i};\mathcal{F})$} is defined by the following:
\begin{eqnarray*}
g_{i}(X,L_{1},\dots , L_{n-i};\mathcal{F})
&=&(-1)^{i}(\chi_{i}^{H}(X,L_{1},\dots , L_{n-i})-\chi(\mathcal{F})) \\
&&\ \ \ +\sum_{j=0}^{n-i}(-1)^{n-i-j}h^{n-j}(\mathcal{F}).
\end{eqnarray*}
\end{Definition}

\begin{Definition}\label{B2-2}
Let $X$ be a projective variety of dimension $n$, let $i$ be an integer with $0\leq i\leq n-1$, and let $L_{1},\dots , L_{n-i}$ be line bundles on $X$.
Then we set
\begin{eqnarray*}
\chi_{i}^{H}(X,L_{1},\dots , L_{n-i})
&:=&\chi_{i}^{H}(X,L_{1},\dots , L_{n-i};{\mathcal O}_{X}),\\
g_{i}(X,L_{1},\dots , L_{n-i})
&:=&g_{i}(X,L_{1},\dots , L_{n-i};{\mathcal O}_{X}), \\
p_{a}^{i}(X,L_{1},\dots , L_{n-i})
&:=&p_{a}^{i}(X,L_{1},\dots , L_{n-i};{\mathcal O}_{X}).
\end{eqnarray*}
\end{Definition}

\begin{Remark}\label{R4}
\begin{itemize}
\item [(1)]
We can prove that $\chi_{p_{1},\cdots, p_{n-i}}(L_{1},\dots, L_{n-i};\mathcal{F})$ is an integer for every non-negative integers $p_{1},\dots , p_{n-i}$ with $0\leq p_{1}+\cdots +p_{n-i}\leq n$.
So in particular we see that $g_{i}(X,L_{1},\dots , L_{n-i};\mathcal{F})$ is an integer.
\item [(2)]
If $i=0$, then
$g_{0}(X,L_{1},\dots , L_{n})=L_{1}\cdots L_{n}$.
\item [(3)]
If $L_{1}=\cdots =L_{n-i}=L$, then $g_{i}(X,\underbrace{L, \dots , L}_{n-i})=g_{i}(X,L)$.
(Here $g_{i}(X,L)$ denotes the $i$th sectional geometric genus of $(X,L)$ (see \cite[Definition 2.1]{Fukuma04}).)
In particular, 
if $i=n-1$, then $g_{n-1}(X,L_{1})$ in Definition \ref{B2-2} is equal to the $(n-1)$th sectional geometric genus of $(X,L_{1})$ in \cite[Definition 2.1]{Fukuma04}.

\item [(4)]
If $i=n$, then $g_{n}(X)=h^{n}({\mathcal O}_{X})$.
\item [(5)]
$\chi_{i}^{H}(X,L_{1},\dots , L_{n-i};\mathcal{F})$ in Definition \ref{B2} (1) is called the {\it $i$th sectional $H$-arithmetic genus of $(X,L_{1},\dots, L_{n-i};\mathcal{F})$.}
\item [(6)]
Let $X$ be a {\it smooth} projective variety of dimension $n$ and let $\mathcal{E}$ be an ample vector bundle of rank $r$ on $X$ with $1\leq r\leq n$.
Then in \cite[Definition 2.1]{Fukuma03}, we defined the 
{\it $i$th $c_{r}$-sectional geometric genus $g_{i}(X,\mathcal{E})$ of $(X,\mathcal{E})$} for every integer $i$ with $0\leq i\leq n-r$.
Let $i$ be an integer with $0\leq i\leq n-1$ and let $L_{1}, \dots , L_{n-i}$ be ample line bundles on $X$.
By setting $\mathcal{E}:=L_{1}\oplus \cdots \oplus L_{n-i}$,
we see that $g_{i}(X,\mathcal{E})=g_{i}(X,L_{1},\dots , L_{n-i})$.
\end{itemize}
\end{Remark}

\begin{Proposition}\label{B18}
Let $X$ be a projective variety of dimension $n$ and 
let $i$ be an integer with $0\leq i\leq n-1$.
Let $A, B, L_{1}, \cdots , L_{n-i-1}$ be line bundles on $X$.
Then 
\begin{eqnarray*}
&&g_{i}(X,A+B,L_{1},\cdots , L_{n-i-1}) \\
&&=g_{i}(X,A,L_{1},\cdots , L_{n-i-1})+g_{i}(X,B,L_{1},\cdots , L_{n-i-1}) \\
&&\ \ \ +g_{i-1}(X,A,B,L_{1},\cdots , L_{n-i-1})-h^{i-1}(\mathcal{O}_{X}). 
\end{eqnarray*}
\end{Proposition}
\begin{proof}
See \cite[Corollary 2.4 and Remark 2.6]{Fukuma08}.
\end{proof}

In \cite[Theorem 5.1]{Fukuma08-3} we obtained the following equality under the assumption that $X$ is smooth.
But by the same argument as in the proof of \cite{Fukuma08-3}, we can also prove this equality
if $X$ is a normal Gorenstein projective variety with $\dim X=n\geq 2$ such that $X$ has at most terminal singularities
because the Serre duality and the Kawamata-Viehweg vanishing theorem hold in this case.

\begin{Theorem}\label{I1}
Let $X$ be a normal Gorenstein projective variety with $\dim X=n\geq 2$ such that $X$ has at most terminal singularities,
let $L_{1},\cdots ,L_{m}$ be nef and big line bundles on $X$ and let $L$ be a nef line bundle, where $m\geq 1$.
Then
\begin{eqnarray*}
&&h^{0}(K_{X}+L_{1}+\cdots +L_{m}+L)-h^{0}(K_{X}+L_{1}+\cdots +L_{m})\\
&&=\sum_{s=0}^{n-1}\sum_{(k_{1},\cdots,k_{n-s-1})\in A_{n-s-1}^{m}}
g_{s}(X, L_{k_{1}},\cdots, L_{k_{n-s-1}},L) \\
&&\ \ \ -\sum_{s=0}^{n-2}{m-1\choose n-s-2}h^{s}(\mathcal{O}_{X}).
\end{eqnarray*}
Here $A_{t}^{p}:=\left\{ (k_{1},\cdots , k_{t})\ |\ k_{l}\in \{ 1, \cdots , p\}, k_{i}<k_{j} \ \mbox{\rm if $i<j$}\right\}$, and we set
\[
\sum_{(k_{1},\cdots,k_{n-s-1})\in A_{n-s-1}^{m}}
g_{s}(X, L_{k_{1}},\cdots, L_{k_{n-s-1}},L)
=\left\{
\begin{array}{lc}
0 & \mbox{if $n-s-1>m$,} \\
g_{n-1}(X,L) & \mbox{if $s=n-1$.}
\end{array} \right. \]
\end{Theorem}

\begin{Definition}
Let $(\sharp)$ be an assumption of polarized varieties $(X,L)$.
For any fixed positive integer $n$, we set
\begin{eqnarray*}
\mathcal{P}_{n}(\sharp)
&:=&\left \{\ \mbox{\rm $(X,L)$ : polarized variety}\ |\  \mbox{\rm $\dim X=n$, $(X,L)$ satisfies $(\sharp)$}\right.\\
&&\ \ \ \ \ \ \ \ \ \ \ \ \ \ \ \ \ \ \ \ \ \ \ \ \ \ \ \ \ \ \ \ \ \ \ \ \ \ 
\left. \mbox{ and $\kappa(K_{X}+L)\geq 0$}\right\}, \\
\mathcal{P}_{n}^{{\rm NEF}}(\sharp)
&:=&\left \{\ \mbox{\rm $(X,L)$ : polarized variety}\ |\  \mbox{\rm $\dim X=n$, $(X,L)$ satisfies $(\sharp)$}\right.\\
&&\ \ \ \ \ \ \ \ \ \ \ \ \ \ \ \ \ \ \ \ \ \ \ \ \ \ \ \ \ \ \ \ \ \ \ \ \ \ 
\left. \mbox{ and $K_{X}+L$ is nef}\right\}, \\
&& \\
\mathcal{M}_{n}(\sharp)
&:=&\left\{\ r\in\mathbb{N}\ |\ h^{0}(r(K_{X}+L))>0\  \mbox{\rm for any $(X,L)\in{\mathcal{P}}_{n}(\sharp)$}\right\},\\
\mathcal{M}_{n}(\sharp)^{+}
&:=&\left\{\ r\in\mathbb{N}\ |\ h^{0}(m(K_{X}+L))>0\  \mbox{\rm for any $m\geq r$ and any $(X,L)\in{\mathcal{P}}_{n}(\sharp)$}\right\},\\
\mathcal{M}_{n}^{{\rm NEF}}(\sharp)
&:=&\left\{\ r\in\mathbb{N}\ |\ h^{0}(r(K_{X}+L))>0\  \mbox{\rm for any $(X,L)\in{\mathcal{P}}_{n}^{{\rm NEF}}(\sharp)$}\right\},\\
\mathcal{M}_{n}^{{\rm NEF}}(\sharp)^{+}
&:=&\left\{\ r\in\mathbb{N}\ |\ h^{0}(m(K_{X}+L))>0\  \mbox{\rm for any $m\geq r$ and any $(X,L)\in{\mathcal{P}}_{n}^{{\rm NEF}}(\sharp)$}\right\},\\
&& \\
m_{n}(\sharp)&:=&
\left\{
\begin{array}{ll}
\mbox{\rm {min}}\ \mathcal{M}_{n}(\sharp)
& \ \ \mbox{\rm if $\mathcal{M}_{n}(\sharp)\neq \emptyset$,} \\
\infty & \ \ \mbox{\rm if $\mathcal{M}_{n}(\sharp)=\emptyset$.}
\end{array}\right. \\
m_{n}(\sharp)^{+}&:=&
\left\{
\begin{array}{ll}
\mbox{\rm {min}}\ \mathcal{M}_{n}(\sharp)^{+}
& \ \ \mbox{\rm if $\mathcal{M}_{n}(\sharp)^{+}\neq \emptyset$,} \\
\infty & \ \ \mbox{\rm if $\mathcal{M}_{n}(\sharp)^{+}=\emptyset$.}
\end{array}\right. \\
m_{n}^{{\rm NEF}}(\sharp)&:=&
\left\{
\begin{array}{ll}
\mbox{\rm {min}}\ \mathcal{M}_{n}^{{\rm NEF}}(\sharp)
& \ \ \mbox{\rm if $\mathcal{M}_{n}^{{\rm NEF}}(\sharp)\neq \emptyset$,} \\
\infty & \ \ \mbox{\rm if $\mathcal{M}_{n}^{{\rm NEF}}(\sharp)=\emptyset$.}
\end{array}\right. \\
m_{n}^{{\rm NEF}}(\sharp)^{+}&:=&
\left\{
\begin{array}{ll}
\mbox{\rm {min}}\ \mathcal{M}_{n}^{{\rm NEF}}(\sharp)^{+}
& \ \ \mbox{\rm if $\mathcal{M}_{n}^{{\rm NEF}}(\sharp)^{+}\neq \emptyset$,} \\
\infty & \ \ \mbox{\rm if $\mathcal{M}_{n}^{{\rm NEF}}(\sharp)^{+}=\emptyset$.}
\end{array}\right. 
\end{eqnarray*}
\end{Definition}

\begin{Remark}\label{TREM-1}
Here we note that the following inequality hold.
\begin{eqnarray*}
m_{n}(\sharp)&\leq& m_{n}(\sharp)^{+},\\
m_{n}^{{\rm NEF}}(\sharp)&\leq& m_{n}^{{\rm NEF}}(\sharp)^{+},\\
m_{n}^{{\rm NEF}}(\sharp)&\leq& m_{n}(\sharp),\\
m_{n}^{{\rm NEF}}(\sharp)^{+}&\leq& m_{n}(\sharp)^{+}.
\end{eqnarray*}
\end{Remark}

\begin{Lemma}\label{Lemma B}
Let $X$ be a complete normal variety of dimension $n$, 
and let $D_{1}$ and $D_{2}$ be effective Cartier divisors on $X$.
Then $h^{0}(D_{1}+D_{2})\geq h^{0}(D_{1})+h^{0}(D_{2})-1$.
\end{Lemma}
\begin{proof}
See \cite[Lemma 1.12]{Fukuma04} or \cite[15.6.2 Lemma]{Kollar95}.
\end{proof}

\begin{Lemma}\label{Lemma C}
Let $p$ and $q$ be positive integers such that $p$ and $q$ are coprime.
Then for any integer $l$ with $l\geq (p-1)(q-1)$, there exist
non-negative integers $i$ and $j$ such that $l=pi+qj$.
\end{Lemma}
\begin{proof}
We note that there exists a pair of integers $(\alpha, \beta)$ such that $p\alpha+q\beta=l$.
Then we can easily see that any integers $x$ and $y$ which satisfy $px+qy=l$ can be expressed as $x=\alpha+qm$ and $y=\beta-pm$, where $m$ is an arbitrary integer.
In particular there exists a pair of integers $(x_{1},y_{1})$ with $px_{1}+qy_{1}=l$ and $0\leq x_{1}<q$.
If $y_{1}\geq 0$, then we get the assertion.
So we may assume that $y_{1}<0$.
Then $l=px_{1}+qy_{1}\leq px_{1}-q\leq p(q-1)-q=(p-1)(q-1)-1$.
But this is a contradiction because we assume that $l\geq (p-1)(q-1)$.
\end{proof}

\begin{Lemma}\label{Lemma D}
Let $X$ be a smooth projective variety of dimension $n\geq 4$ and let $V$ be a normal projective variety of dimension $n\geq 4$ with $\dim \mbox{\rm Sing}(V)\leq n-4$.
Let $\pi: X\to V$ be a birational morphism such that $X\backslash \pi^{-1}(\mbox{\rm Sing}(V))\cong V\backslash \mbox{\rm Sing}(V)$.
Let $E$ be a $\pi$-exceptional irreducible and reduced divisor on $X$, $A_{1}$ and $A_{2}$ line bundles on $X$ and $L_{1}, \dots , L_{n-3}$ line bundles on $V$.
Then $EA_{1}A_{2}(\pi^{*}(L_{1}))\cdots (\pi^{*}(L_{n-3}))=0$.
\end{Lemma}
\begin{proof}
By \cite[Proposition 4 in section 2, chapter I]{Kleiman66}, we have
$$EA_{1}A_{2}(\pi^{*}(L_{1}))\cdots (\pi^{*}(L_{n-3}))
=(A_{1}|_{E})(A_{2}|_{E})(\pi^{*}(L_{1}))|_{E}\cdots (\pi^{*}(L_{n-3}))|_{E}.$$
On the other hand, since $\dim \mbox{Sing}V\leq n-4$, we have $\dim\pi(E)\leq n-4$.
Here we set $Z:=\pi(E)$.
Then
$$(A_{1}|_{E})(A_{2}|_{E})(\pi^{*}(L_{1}))|_{E}\cdots (\pi^{*}(L_{n-3}))|_{E}
=(A_{1}|_{E})(A_{2}|_{E})((\pi|_{E})^{*}(L_{1}|_{Z}))\cdots ((\pi|_{E})^{*}(L_{n-3}|_{Z})).$$
Here we set 
$$f(t_{1}, \dots , t_{n-1}):=\chi(E, ((\pi|_{E})^{*}(L_{1}|_{Z}))^{\otimes t_{1}}\otimes \cdots \otimes ((\pi|_{E})^{*}(L_{n-3}|_{Z}))^{\otimes t_{n-3}}\otimes (A_{1}|_{E})^{\otimes t_{n-2}}\otimes (A_{2}|_{E})^{\otimes t_{n-1}}).$$
Then $f(t_{1}, \dots , t_{n-1})$ is a polynomial of $t_{1}, \dots , t_{n-1}$ of degree at most $n-1$.
Let $C_{1}$ (resp. $C_{2}$, $C_{3}$, $C_{4}$, $C_{5}$, $C_{6}$) be the coefficient of $t_{1}\cdots t_{n-3}$ (resp. $t_{1}\cdots t_{n-3}t_{n-2}$, $t_{1}\cdots t_{n-3}t_{n-1}$, $t_{1}\cdots t_{n-3}t_{n-2}^{2}$, $t_{1}\cdots t_{n-3}t_{n-1}^{2}$, $t_{1}\cdots t_{n-3}t_{n-2}t_{n-1}$) in $f(t_{1}, \dots , t_{n-1})$.
\\
Then $f(t_{1}, \dots , t_{n-3}, 0, 0)=\chi(E, ((\pi|_{E})^{*}(L_{1}|_{Z}))^{\otimes t_{1}}\otimes \cdots \otimes ((\pi|_{E})^{*}(L_{n-3}|_{Z}))^{\otimes t_{n-2}})$.
Here we set 
$$g(t_{1}, \dots , t_{n-3}):=f(t_{1}, \dots , t_{n-3}, 0, 0).$$
Then the coefficient of $t_{1}\cdots t_{n-3}$ in $g(t_{1}, \dots , t_{n-3})$ 
is equal to $C_{1}$.
On the other hand since the degree of $g(t_{1}, \dots , t_{n-3})$ is less than $n-3$
(see the proof of \cite[Proposition 6 in section 2, chapter I]{Kleiman66}), 
we have $C_{1}=0$.
\par
Next we consider the polynomial $f(t_{1}, \dots , t_{n-3}, 1, 0)$ (resp. $f(t_{1}, \dots , t_{n-3}, 0, 1)$).
Then the coefficient of $t_{1}\cdots t_{n-3}$ in $f(t_{1}, \dots , t_{n-3}, 1, 0)$ (resp. $f(t_{1}, \dots , t_{n-3}, 0, 1)$) is $C_{1}+C_{2}+C_{4}$ (resp. $C_{1}+C_{3}+C_{5}$).
Moreover 
$$f(t_{1}, \dots , t_{n-3}, 1, 0)=\chi(E, ((\pi|_{E})^{*}(L_{1}|_{Z}))^{\otimes t_{1}}\otimes \cdots \otimes ((\pi|_{E})^{*}(L_{n-3}|_{Z}))^{\otimes t_{n-2}}\otimes (A_{1}|_{E}))$$ 
and the degree of this polynomial is less than $n-3$.
Hence $C_{1}+C_{2}+C_{4}=0$.
By the same reason as this, we have $C_{1}+C_{3}+C_{5}=0$.
Therefore $C_{2}+C_{4}=C_{3}+C_{5}=0$ since $C_{1}=0$.
\par
Finally we consider $f(t_{1}, \dots , t_{n-3}, 1, 1)$.
Then the coefficient of $t_{1}\cdots t_{n-3}$ in $f(t_{1}, \dots , t_{n-3}, 1, 1)$ is $C_{1}+C_{2}+C_{3}+C_{4}+C_{5}+C_{6}$.
Moreover 
$$f(t_{1}, \dots , t_{n-3}, 1, 1)=\chi(E, ((\pi|_{E})^{*}(L_{1}|_{Z}))^{\otimes t_{1}}\otimes \cdots \otimes ((\pi|_{E})^{*}(L_{n-3}|_{Z}))^{\otimes t_{n-2}}\otimes (A_{1}|_{E})\otimes (A_{2}|_{E}))$$ 
and the degree of this polynomial is less than $n-3$.
Hence $C_{1}+C_{2}+C_{3}+C_{4}+C_{5}+C_{6}=0$.
Therefore $C_{6}=0$ because $C_{1}=C_{2}+C_{4}=C_{3}+C_{5}=0$.
\par
By above we see that the coefficient of $t_{1}\cdots t_{n-1}$ in $f(t_{1},\dots , t_{n-1})$ is zero.
Therefore by the definition of intersection numbers (see \cite{Kleiman66}) we have
$(A_{1}|_{E})(A_{2}|_{E})(\pi|_{E})^{*}(L_{1}|_{Z})\cdots (\pi|_{E})^{*}(L_{n-3}|_{Z})=0$.
Hence we get the assertion. 
\end{proof}

\begin{Lemma}\label{Lemma E}
Let $X$ be a smooth projective variety of dimension $n\geq 4$ and let $V$ be a normal Gorenstein projective variety of dimension $n\geq 4$ with only terminal singularities and $\dim \mbox{\rm Sing}(V)\leq n-4$.
Let $\pi: X\to V$ be a birational morphism such that $X\backslash \pi^{-1}(\mbox{\rm Sing}(V))\cong V\backslash \mbox{\rm Sing}(V)$.
Let $E_{\pi}$ be the $\pi$-exceptional divisor on $X$ with $K_{X}=\pi^{*}(K_{V})+E_{\pi}$ and $L_{1}, \dots , L_{n-3}$ line bundles on $V$.
Then $c_{2}(X)E_{\pi}(\pi^{*}(L_{1}))\cdots (\pi^{*}(L_{n-3}))=0$.
\end{Lemma}
\begin{proof}
Let $A_{1}, \dots, A_{n-3}$ and $A$ be line bundles on $X$.
Then by \cite[Theorem 2.4]{Fukuma08} we get the following.
\begin{eqnarray}
&&\chi_{3}^{H}(X,A_{1}, \dots , A_{n-3};A) \label{Lemma E-1}\\
&&=\sum_{k=0}^{3}\left(\sum_{(t_{1}, \cdots, t_{n-3})\in S(n-3)^{+}_{n-k}}
\frac{(-1)^{k}}{(t_{1})!\cdots (t_{n-3})!}
A_{1}^{t_{1}}\cdots A_{n-3}^{t_{n-3}}\right)R_{k}(X,A). \nonumber
\end{eqnarray}

Here we note that
\begin{eqnarray*}
R_{0}(X,A)&=&1,\\
R_{1}(X,A)&=&T_{1}(X)+\mbox{\rm ch}(A)_{1}=\frac{1}{2}c_{1}(X)+A,\\
R_{2}(X,A)&=&T_{2}(X)+\mbox{\rm ch}(A)_{1}T_{1}(X)+\mbox{\rm ch}(A)_{2}\\
          &=&\frac{1}{12}(c_{2}(X)+c_{1}(X)^{2})+\frac{1}{2}c_{1}(X)A+\frac{1}{2}A^{2},\\
R_{3}(X,A)&=&T_{3}(X)+\mbox{\rm ch}(A)_{1}T_{2}(X)+\mbox{\rm ch}(A)_{2}T_{1}(X)+\mbox{\rm ch}(A)_{3}\\
          &=&\frac{1}{24}c_{1}(X)c_{2}(X)
          +\frac{1}{12}(c_{2}(X)+c_{1}(X)^{2})A+\frac{1}{4}c_{1}(X)A^{2}+\frac{1}{6}A^{3}.
\end{eqnarray*}

Here we put $A_{i}=\pi^{*}(L_{i})$ for $i=1,\dots , n-3$.
Then by using Lemma \ref{Lemma D} and the equation (\ref{Lemma E-1}) above we have
\begin{eqnarray}
&&\chi_{3}^{H}(X,\pi^{*}(L_{1}), \dots , \pi^{*}(L_{n-3}); \mathcal{O}(E_{\pi}))
-\chi_{3}^{H}(X,\pi^{*}(L_{1}), \dots , \pi^{*}(L_{n-3}); \mathcal{O}_{X}) \label{Lemma E-2}\\
&&=\frac{1}{12}c_{2}(X)E_{\pi}\pi^{*}(L_{1})\cdots\pi^{*}(L_{n-3}). \nonumber 
\end{eqnarray}

By Grauert-Riemenschneider's theorem we see that for every $i\geq 1$
\begin{eqnarray*}
0&=&R^{i}\pi_{*}(\mathcal{O}(K_{X}))\\
&=&R^{i}\pi_{*}(\pi^{*}(K_{V})+E_{\pi})\\
&=&R^{i}\pi_{*}(\mathcal{O}(E_{\pi}))\otimes K_{V}.
\end{eqnarray*}
Hence $R^{i}\pi_{*}(\mathcal{O}(E_{\pi}))=0$ for every $i\geq 1$.
We also note that $R^{i}\pi_{*}(\mathcal{O}_{X})=0$ for every $i\geq 1$ because $V$ has only rational singularities.
So we see that for every integer $i$ with $i\geq 0$ we have 
\begin{eqnarray*}
&&h^{i}(\pi^{*}(L_{1})^{\otimes t_{1}}\otimes \cdots \otimes \pi^{*}(L_{n-3})^{\otimes t_{n-3}}\otimes \mathcal{O}(E_{\pi}))\\
&&=h^{i}((L_{1})^{\otimes t_{1}}\otimes \cdots \otimes (L_{n-3})^{\otimes t_{n-3}}\otimes \pi_{*}(\mathcal{O}(E_{\pi})))\\
&&=h^{i}((L_{1})^{\otimes t_{1}}\otimes \cdots \otimes (L_{n-3})^{\otimes t_{n-3}})\\
&&=h^{i}(\pi^{*}(L_{1})^{\otimes t_{1}}\otimes \cdots \otimes \pi^{*}(L_{n-3})^{\otimes t_{n-3}}).
\end{eqnarray*}

Therefore
\begin{eqnarray*}
&&\chi(\pi^{*}(L_{1})^{\otimes t_{1}}\otimes \cdots \otimes \pi^{*}(L_{n-3})^{\otimes t_{n-3}}\otimes \mathcal{O}(E_{\pi}))\\
&&=\chi(\pi^{*}(L_{1})^{\otimes t_{1}}\otimes \cdots \otimes \pi^{*}(L_{n-3})^{\otimes t_{n-3}}).
\end{eqnarray*}

In particular, we have
\begin{eqnarray}
&&\chi_{3}^{H}(X,\pi^{*}(L_{1}), \dots , \pi^{*}(L_{n-3}); \mathcal{O}(E_{\pi})) \label{Lemma E-3}\\
&&=\chi_{3}^{H}(X,\pi^{*}(L_{1}), \dots , \pi^{*}(L_{n-3}); \mathcal{O}_{X}). \nonumber
\end{eqnarray}

So by (\ref{Lemma E-2}) and (\ref{Lemma E-3}) we get the assertion. 
\end{proof}

\section{The case where \textrm{\boldmath $K_{X}+L$} is nef}

In this section, we assume that $(X,L)$ satisfies the following assumption (SRE).
\\
\\
(SRE): $(X,L)$ is a polarized variety of dimension $n$ such that $X$ is a normal Gorenstein projective variety with only isolated terminal singularities.
\\
\par
Here we note that this condition appears when we take the second reduction of polarized manifolds of even dimension.
\par
First we prove the following proposition.

\begin{Proposition}\label{A-P1}
$\mathcal{M}_{n}^{NEF}(\mbox{\rm SRE})\neq \emptyset$
\end{Proposition}
\begin{proof}
Let $(X,L)\in \mathcal{P}_{N}^{NEF}(\mbox{\rm SRE})$.
Since $(m-1)K_{X}+mL$ is nef and big for every integer $m\geq 1$, we have
$h^{i}(m(K_{X}+L))=0$ for every integer $i\geq 1$ by the Kawamata-Viehweg vanishing theorem.
Hence $h^{0}(t(K_{X}+L))=\chi(t(K_{X}+L))$ for every integer $t\geq 1$.
Since $\chi(t(K_{X}+L))$ is a polynomial in $t$ of degree at most $n$, there exists an integer $p$ such that $1\leq p\leq n+1$ and $h^{0}(p(K_{X}+L))>0$.
Using Lemma \ref{Lemma B}, we have $h^{0}((n+1)!(K_{X}+L))>0$ for any $(X,L)\in \mathcal{P}_{N}^{NEF}(\mbox{\rm SRE})$. Therefore $(n+1)!\in \mathcal{M}_{n}^{NEF}(\mbox{\rm SRE})$ and we get the assertion. 
\end{proof}

Next we will prove the following theorem.

\begin{Theorem}\label{Theorem1.6}
Let $(X,L)$ be a polarized variety of dimension $n\geq 4$ which satisfies the assumption {\rm (SRE)}, and let $Y$ be a normal projective variety of dimension $3$.
Assume that there exists a fiber space $f:X\to Y$ such that $K_{X}+L=f^{*}(H)$ for some ample line bundle $H$ on $Y$.
Then the following hold:
\begin{itemize}
\item [\rm (1)]
If $h^{1}(\mathcal{O}_{X})\geq 1$, then $h^{0}(m(K_{X}+L))\geq 1$ for every positive integer $m$.
\item [\rm (2)]
If $h^{1}(\mathcal{O}_{X})=0$, then $h^{0}(m(K_{X}+L))\geq 1$ for every integer $m$ with $m\geq 2$.
\end{itemize}
\end{Theorem}
\begin{proof}
Let $\delta: T\to Y$ be a resolution of $Y$ such that $T\setminus \delta^{-1}(\mbox{Sing}(Y))\cong Y\setminus \mbox{Sing}(Y)$.
Then there exist a smooth projective variety $X^{\prime}$, a birational morphism $\mu: X^{\prime}\to X$ and a fiber space $f^{\prime}: X^{\prime}\to T$ such that $f\circ \mu=\delta\circ f^{\prime}$.
\\
(1) The case where $h^{1}(\mathcal{O}_{X})>0$.
\\
(1.1) First we note that $h^{1}(\mathcal{O}_{X})=h^{1}(\mathcal{O}_{Y})\leq h^{1}(\mathcal{O}_{T})\leq h^{1}(\mathcal{O}_{X^{\prime}})=h^{1}(\mathcal{O}_{X})$.
Hence $h^{1}(\mathcal{O}_{Y})=h^{1}(\mathcal{O}_{T})$, and 
by \cite[Lemma 0.3.3]{Sommese86} or \cite[Lemma 2.4.1 and Remark 2.4.2]{BeSo-Book},
we see that $Y$ has the Albanese map.
Let $\alpha: Y\to \mbox{Alb}(Y)$ be the Albanese map of $Y$
and let $h:=\alpha\circ f\circ \mu$.
\\
(a) First we consider the case where $\dim h(X^{\prime})=3$.
By \cite[Corollary 10.7 in Chapter III]{Hartshorne}
any general fiber $F_{h}$ of $h$ can be written as follows:
$F_{h}=\cup_{i=1}^{r}F_{i}$, where $F_{i}$ is a smooth projective variety 
of dimension $n-3$.
We note that $F_{i}$ is a fiber of $f\circ \mu$ and we can take $\mu$ such that $K_{X^{\prime}}|_{F_{h}}=\mu^{*}(K_{X})|_{F_{h}}$ because $\dim \mbox{Sing}(X)\leq 0$ and $\dim \mbox{Sing}(Y)\leq 1$.
Hence 
\begin{eqnarray*}
h^{0}((K_{X^{\prime}}+\mu^{*}(L))|_{F_{h}})
&=&\sum_{i=1}^{r}h^{0}(\mu^{*}(f^{*}(H))|_{F_{i}}) \\
&=&\sum_{i=1}^{r}h^{0}(\mathcal{O}_{F_{i}}) \\
&>&0.
\end{eqnarray*}
By \cite[Lemma 4.1]{1.5} we have $h^{0}(K_{X}+L)=h^{0}(K_{X^{\prime}}+\mu^{*}(L))>0$.
\\
(b) Next we consider the case of $\dim h(X^{\prime})=2$.
Then $\dim \alpha(Y)=2$ and let $Y\to Z\to \alpha(Y)$ be the Stein factorization of $\alpha$.
We set $\alpha_{1}:Y\to Z$, $\delta_{1}:=\alpha_{1}\circ \delta$ and $h_{1}:=\delta_{1}\circ f^{\prime}$.
Then we note that $h_{1}$ has connected fibers.
Let $F_{h_{1}}$ (resp. $F_{\delta_{1}}$) be a general fiber of $h_{1}$ (resp. $\delta_{1}$).
As in the case (a) above, we can take $\mu$ such that $K_{X^{\prime}}|_{F_{h_{1}}}=\mu^{*}(K_{X})|_{F_{h_{1}}}$ because $\dim \mbox{Sing}(X)\leq 0$ and $\dim \mbox{Sing}(Y)\leq 1$.
Then $F_{h_{1}}$ and $F_{\delta_{1}}$ are smooth with $\dim F_{h_{1}}=n-2$ and $\dim F_{\delta_{1}}=1$, $f|_{F_{h_{1}}}:F_{h_{1}}\to F_{\delta_{1}}$ is a fiber space such that
$K_{F_{h_{1}}}+(\mu^{*}L)|_{F_{h_{1}}}=(\mu^{*}\circ f^{*}(H))|_{F_{h_{1}}}=(f^{\prime}|_{F_{h_{1}}})^{*}(\delta^{*}(H)|_{F_{\delta_{1}}})$.
Here we note that $\delta^{*}(H)|_{F_{\delta_{1}}}$ is ample because $\dim F_{\delta_{1}}=1$ and $\deg\delta^{*}(H)|_{F_{\delta_{1}}}>0$.
By \cite[Theorem 4.1]{Fukuma10}, we have $h^{0}(K_{F_{h_{1}}}+\mu^{*}(L)|_{F_{h_{1}}})>0$.
Therefore by \cite[Lemma 4.1]{1.5} we get $h^{0}(K_{X}+L)=h^{0}(K_{X^{\prime}}+\mu^{*}(L))>0$.
\\
(c) Next we consider the case of $\dim h(X^{\prime})=1$.
Then $\alpha(Y)$ is a smooth curve and $\alpha: Y\to \alpha(Y)$ has connected fibers (see \cite[Lemma 2.4.5]{BeSo-Book}).
Let $F_{h}$ (resp. $F_{\alpha}$) be a general fiber of $h$ (resp. $\alpha$).
Then $F_{h}$ is smooth and $F_{\alpha}$ is a projective variety with $\dim F_{\alpha}=2$
and $(f\circ \mu)|_{F_{h}}:F_{h}\to F_{\alpha}$ is a surjective morphism with connected fibers.
By taking its Stein factorization, if necessary, we may assume that $F_{\alpha}$ is normal.
Since $K_{F_{h}}+L_{F_{h}}=\mu^{*}(f^{*}(H))|_{F_{h}}=((f\circ\mu)|_{F_{h}})^{*}(H|_{F_{\alpha}})$ and $H|_{F_{\alpha}}$ is ample,
by \cite[Theorem 4.3]{Fukuma10} we see that $h^{0}(K_{F_{h}}+L_{F_{h}})>0$.
Therefore by \cite[Lemma 4.1]{1.5} we get $h^{0}(K_{X}+L)=h^{0}(K_{X^{\prime}}+\mu^{*}(L))>0$.
\par
From (a), (b) and (c) above, we get $h^{0}(K_{X}+L)>0$.
Therefore we see that $h^{0}(m(K_{X}+L))>0$ by Lemma \ref{Lemma B}.
\\
\\
(2) The case where $h^{1}(\mathcal{O}_{X})=0$.
\par
Then $h^{1}(\mathcal{O}_{Y})=0$.
If $h^{0}(K_{X}+L)>0$, then we get the assertion by Lemma \ref{Lemma B}.
So we assume that $h^{0}(K_{X}+L)=0$.
\par
Since $R^{i}f_{*}(p(K_{X}+L))=R^{i}f_{*}(K_{X}+((p-1)K_{X}+pL))=0$ for every integers $i$ and $p$ with $i>0$ and $p>0$,
we have $h^{i}(p(K_{K}+L))=h^{i}(f_{*}(p(K_{X}+L)))=h^{i}(pH)$ for every integer $i$ and $p$ with $i\geq 0$ and $p>0$.
Therefore $\chi(pH)=h^{0}(pH)$ for every positive integer $p$.
Since $h^{0}(K_{X}+L)=0$, we get $\chi(H)=0$.
Let $t$ be an indeterminate.
Then $\chi(tH)$ is a polynomial of $t$ whose degree is $3$.
Because $\chi(H)=0$, 
we can write $\chi(tH)=d(t-1)(t^{2}+at+b)$, where $a$ and $d$ are real numbers.
On the other hand, we set 
$$\chi(tH)=\sum_{j=0}^{3}\chi_{j}(Y,H){t+j-1\choose j}$$
Then $\chi_{3}(Y,H)=6d$, $\chi_{2}(Y,H)+\chi_{3}(Y,H)=2d(a-1)$,
$2\chi_{3}(Y,H)+3\chi_{2}(Y,H)+6\chi_{1}(Y,H)=6d(b-a)$ 
and $\chi_{0}(Y,H)=-bd$.
Since $H^{3}=\chi_{3}(Y,H)=6d$, we have $g_{1}(Y,H)=1-\chi_{2}(Y,H)=1-2d(a-4)$
and $g_{2}(Y,H)=-1+h^{1}(\mathcal{O}_{Y})+\chi_{1}(Y,H)=d(b-2a+2)-1$.
\par
Next we prove the following claim.
\begin{Claim}\label{CL-T1}
$a\geq -1/2$.
\end{Claim}
\begin{proof}
Let $\delta: T\to Y$ be a resolution of $Y$ such that $T\setminus \delta^{-1}(\mbox{Sing}(Y))\cong Y\setminus \mbox{Sing}(Y)$.
Then there exist a smooth projective variety $X^{\prime}$, a birational morphism $\mu: X^{\prime}\to X$ and a fiber space $f^{\prime}:X^{\prime}\to T$ such that $f\circ \mu=\delta\circ f^{\prime}$.
Let $L^{\prime}=\mu^{*}(L)$.
By the same argument as in the proof of \cite[Theorem 4.3]{Fukuma10}, we have $0\leq (\delta^{*}(H))^{3}-K_{T}(\delta^{*}(H))^{2}$ because $K_{X^{\prime}/T}+L^{\prime}$ is pseudo-effective.
On the other hand,
\begin{eqnarray*}
&&(\delta^{*}(H))^{3}-K_{T}(\delta^{*}(H))^{2}\\
&&=3\chi_{0}^{H}(T,\delta^{*}(H))+2\chi_{1}^{H}(T,\delta^{*}(H))\\
&&=3\chi_{0}^{H}(Y,H)+2\chi_{1}^{H}(Y,H)\\
&&=2d(2a+1).
\end{eqnarray*}
Therefore $2d(1+2a)\geq 0$ and $a\geq -1/2$ since $d>0$.
\end{proof}

Assume that $h^{0}(2(K_{X}+L))=0$.
Then $b=-2a-4$ because $\chi(2H)=0$.
Hence by Claim \ref{CL-T1} we have $g_{2}(Y,H)=d(b-2a+2)-1=d(-4a-2)-1\leq -1$.
By \cite[Lemma 3.1]{Fukuma10-2}, we see that 
$\chi_{2}^{H}(T,\delta^{*}(H))\leq \chi_{2}^{H}(Y,H)$ holds.
We also have $h^{1}(\mathcal{O}_{T})=h^{1}(\mathcal{O}_{Y})=0$ since $h^{1}(\mathcal{O}_{X})=0$ and $X$ has only rational singularities.
Therefore we get $g_{2}(T,\delta^{*}(H))\leq g_{2}(Y,H)$.
\par
If $\kappa(T)=-\infty$, then we see that $g_{2}(T,\delta^{*}(H))\geq 0$ by \cite[Proposition 3.1]{Fukuma10-2}.
\par
If $\kappa(T)\geq 0$, then by \cite[(4.2) Theorem]{Fujita89}, there exists a quasi-polarized variety $(V_{1}, H_{1})$ of dimension $3$ such that $V_{1}$ is a normal projective variety with only $\mathbb{Q}$-factorial terminal singularities, $(T,\delta^{*}(H))$ and $(V_{1}, H_{1})$ are birationally equivalent and $K_{V_{1}}+2H_{1}$ is nef.
Let $\pi: X_{1}\to V_{1}$ be a resolution of $V_{1}$.
Then by \cite[Theorem 4.3]{Fukuma10-2} we have
\begin{eqnarray*}
&&g_{2}(T,\delta^{*}(H)) \\
&&\geq -1+\frac{1}{12}\pi^{*}(K_{V_{1}})(\pi^{*}(K_{V_{1}}+2H_{1}))\pi^{*}(H_{1})
-\frac{1}{36}\pi^{*}((K_{V_{1}}+2H_{1}))\pi^{*}(H_{1})^{2}+\frac{1}{9}(\pi^{*}(H_{1}))^{3}.
\end{eqnarray*}
We also note that there exist a smooth projective variety $X_{2}$ of dimension $3$ and birational morphisms $\pi_{2}:X_{2}\to X_{1}$ and $\beta: X_{2}\to T$ such that $\beta^{*}(\delta^{*}(H))=\pi_{2}^{*}(\pi^{*}(H_{1}))$.
Let $\gamma:=\pi\circ\pi_{2}$.
Then $K_{X_{2}}=\beta^{*}(K_{T})+E_{\beta}$ and $K_{X_{2}}=\gamma^{*}(K_{V_{1}})+E_{\gamma}$
hold, where $E_{\beta}$ (resp. $E_{\gamma}$) is a $\beta$-exceptional (resp. $\gamma$-exceptional) effective $\mathbb{Q}$-Cartier divisor.
Since we assume $\kappa(T)\geq 0$, we have
$0\leq K_{X_{2}}\gamma^{*}(K_{V_{1}}+2H_{1})\gamma^{*}(H_{1})=\gamma^{*}(K_{V_{1}})\gamma^{*}(K_{V_{1}}+2H_{1})\gamma^{*}(H_{1})=\pi^{*}(K_{V_{1}})\pi^{*}(K_{V_{1}}+2H_{1})\pi^{*}(H_{1})$.
\par
Moreover since
\begin{eqnarray*}
\pi^{*}(K_{V_{1}}+2H_{1})\pi^{*}(H_{1})^{2}
&=&\gamma^{*}(K_{V_{1}}+2H_{1})\gamma^{*}(H_{1})^{2}\\
&=&(K_{X_{2}}+2\gamma^{*}(H_{1}))\gamma^{*}(H_{1})^{2}\\
&=&(\beta^{*}(K_{T})+E_{\beta}+2\beta^{*}\delta^{*}(H))\beta^{*}\delta^{*}(H)^{2}\\
&=&(K_{T}+2\delta^{*}(H))\delta^{*}(H)^{2}
\end{eqnarray*}
and $(\pi^{*}(H_{1}))^{3}=(\delta^{*}(H))^{3}$,
we have
\begin{eqnarray*}
&&-\frac{1}{36}\pi^{*}(K_{V_{1}}+2H_{1})\pi^{*}(H_{1})^{2}+\frac{1}{9}(\pi^{*}(H_{1}))^{3}\\&&=-\frac{1}{36}K_{T}\delta^{*}(H)^{2}+\frac{1}{18}(\delta^{*}(H))^{3}\\
&&=\frac{1}{36}(4\chi_{0}^{H}(T,\delta^{*}(H))+2\chi_{1}^{H}(T,\delta^{*}(H))\\
&&=\frac{1}{36}(4\chi_{0}^{H}(Y,H)+2\chi_{1}^{H}(Y,H))\\
&&=\frac{1}{18}(2\chi_{0}^{H}(Y,H)+\chi_{1}^{H}(Y,H))\\
&&=\frac{1}{18}(12d+2d(a-1)-6d)\\
&&=\frac{1}{9}(2d+ad).
\end{eqnarray*}

Since $a\geq -1/2$ by Claim \ref{CL-T1}, we have 
$$-\frac{1}{36}\pi^{*}((K_{V_{1}}+2H_{1}))\pi^{*}(H_{1})^{2}+\frac{1}{9}(\pi^{*}(H_{1}))^{3}
\geq \frac{1}{9}(2d-\frac{1}{2}d)=\frac{1}{6}d.$$
Therefore 
$$g_{2}(Y,H)\geq \frac{1}{6}d-1.$$
On the other hand, as we said before, $g_{2}(Y,H)\leq d(-4a-2)-1\leq -1$ holds by Claim \ref{CL-T1}.
But this is impossible because 
$$d=\frac{1}{6}\chi_{0}^{H}(Y,H)=\frac{1}{6}H^{3}>0.$$
Hence $h^{0}(2(K_{X}+L))\neq 0$.
\par
Assume that $h^{0}(3(K_{X}+L))=0$.
Then $b=-3a-9$ because $\chi(3H)=0$.
Hence $g_{2}(Y,H)=d(b-2a+2)-1=d(-5a-7)-1<0$ by Claim \ref{CL-T1}.
But this is impossible by the same argument as above.
Hence $h^{0}(3(K_{X}+L))\neq 0$.
\par
Therefore $h^{0}(m(K_{X}+L))>0$ for every integer $m$ with $m\geq 2$
by Lemmas \ref{Lemma B} and \ref{Lemma C}.
Hence we get the assertion.
\end{proof}

Next we consider the case where 
$X$ is a normal Gorenstein projective variety of dimension $4$ 
with only isolated terminal singularities.

\begin{Theorem}\label{Theorem4.1}
Let $(X,L)$ be a polarized variety of dimension $4$ 
which satisfies the assumption {\rm (SRE)}.
Assume that $K_{X}+L$ is nef.
\begin{itemize}
\item [\rm (1)]
If $0\leq \kappa(K_{X}+L)\leq 2$, then $h^{0}(m(K_{X}+L))>0$ for every integer $m$ with $m\geq 1$.
\item [\rm (2)]
If $\kappa(K_{X}+L)=3$, then $h^{0}(m(K_{X}+L))>0$ for every integer $m$ with $m\geq 2$.
\end{itemize}
\end{Theorem}
\begin{proof}
(i) If $\kappa(K_{X}+L)=0$, 
then we can prove that $\mathcal{O}(K_{X}+L)=\mathcal{O}_{X}$ 
by \cite[Lemma 3.3.2]{BeSo-Book}.
Then $h^{0}(K_{X}+L)=1$.
\\
(ii) If $\kappa(K_{X}+L)=1$ (resp. $2$), then there exist a normal projective variety $Y$ with $\dim Y=1$ (resp. $2$) and a fiber space $f:X\to Y$ such that
$K_{X}+L=f^{*}(H)$ for some ample line bundle $H$ on $Y$.
By the same argument as in the proof of \cite[Theorem 4.1]{Fukuma10} (resp. \cite[Theorem 4.3]{Fukuma10}) we can prove $h^{0}(m(K_{X}+L))>0$ for any $m\geq 1$.
\\
(iii) If $\kappa(K_{X}+L)=3$, then there exist a normal projective variety $Y$ with $\dim Y=3$ and a fiber space $f:X\to Y$ such that
$K_{X}+L=f^{*}(H)$ for some ample line bundle $H$ on $Y$.
By Theorem \ref{Theorem1.6} we get $h^{0}(m(K_{X}+L))>0$ for any $m\geq 2$.
\end{proof}

\begin{Theorem}\label{Theorem4.1-big}
Let $(X,L)$ be a polarized variety of dimension $4$ 
which satisfies the assumption {\rm (SRE)}.
Assume that $K_{X}+L$ is nef and big.
Then $h^{0}(m(K_{X}+L))>0$ for every integer $m$ with $m\geq 4$.
\end{Theorem}
\begin{proof}
(i) First we consider the case of $m=4$.
\begin{Claim}\label{T-CL-1} 
$h^{0}(4(K_{X}+L))>0$.
\end{Claim}
\begin{proof}
Assume that $h^{0}(4(K_{X}+L))=0$.
Then by Lemma \ref{Lemma B}, we get $h^{0}(2(K_{X}+L))=0$.
Therefore we have
\begin{eqnarray*}
0&\geq& h^{0}(2(K_{X}+L))-h^{0}(K_{X}+L), \\
0&\leq& h^{0}(3(K_{X}+L))-h^{0}(2(K_{X}+L)), \\
0&\geq& h^{0}(4(K_{X}+L))-h^{0}(3(K_{X}+L)). 
\end{eqnarray*}
On the other hand by Theorem \ref{I1}
\begin{eqnarray}
&&h^{0}(m(K_{X}+L))-h^{0}((m-1)(K_{X}+L)) \label{Eq}\\
&&=g_{3}(X,K_{X}+L)+g_{2}(X,K_{X}+L,(m-2)K_{X}+(m-1)L)-h^{2}(\mathcal{O}_{X}).
\nonumber
\end{eqnarray}
By using the above,
\begin{eqnarray}
0&\geq& g_{3}(X,K_{X}+L)+g_{2}(X,L,K_{X}+L)-h^{2}(\mathcal{O}_{X}),
\label{Eq4-1}\\
0&\leq& g_{3}(X,K_{X}+L)+g_{2}(X,K_{X}+L,K_{X}+2L)-h^{2}(\mathcal{O}_{X}),
\label{Eq4-2}\\
0&\geq& g_{3}(X,K_{X}+L)+g_{2}(X,K_{X}+L,2K_{X}+3L)-h^{2}(\mathcal{O}_{X}).
\label{Eq4-3}
\end{eqnarray}
By (\ref{Eq4-1}) and (\ref{Eq4-2}) we get
\begin{equation}
g_{2}(X,K_{X}+L,K_{X}+2L)\geq g_{2}(X,L,K_{X}+L). \label{T-Eq1}
\end{equation}
By (\ref{Eq4-2}) and (\ref{Eq4-3}) we get
\begin{equation}
g_{2}(X,K_{X}+L,K_{X}+2L)\geq g_{2}(X,K_{X}+L,2K_{X}+3L). \label{T-Eq2}
\end{equation}
On the other hand by Proposition \ref{B18}
\begin{eqnarray}
g_{2}(X,K_{X}+L,K_{X}+2L)
&=&g_{2}(X,K_{X}+L,K_{X}+L)+g_{2}(X,K_{X}+L,L) \label{T-Eq3}\\
&&+g_{1}(X,K_{X}+L,K_{X}+L,L)-h^{1}(\mathcal{O}_{X}) \nonumber
\end{eqnarray}
and
\begin{eqnarray}
g_{2}(X,K_{X}+L,2K_{X}+3L)
&=&g_{2}(X,K_{X}+L,K_{X}+L)+g_{2}(X,K_{X}+L,K_{X}+2L) \label{T-Eq4}\\
&&+g_{1}(X,K_{X}+L,K_{X}+L,K_{X}+2L)-h^{1}(\mathcal{O}_{X}). \nonumber
\end{eqnarray}

By (\ref{T-Eq1}) and (\ref{T-Eq3}), we have

\begin{equation}
g_{2}(X,K_{X}+L,K_{X}+L)+g_{1}(X,K_{X}+L,K_{X}+L,L)
-h^{1}(\mathcal{O}_{X})\geq 0. \label{T-Eq5}
\end{equation}

By (\ref{T-Eq2}) and (\ref{T-Eq4}), we have

\begin{equation}
g_{2}(X,K_{X}+L,K_{X}+L)+g_{1}(X,K_{X}+L,K_{X}+L,K_{X}+2L)-h^{1}(\mathcal{O}_{X})\leq 0.
 \label{T-Eq6}
\end{equation}

Hence by (\ref{T-Eq5}) and (\ref{T-Eq6}) we get 
$$g_{1}(X,K_{X}+L,K_{X}+L,L)-g_{1}(X,K_{X}+L,K_{X}+L,K_{X}+2L)\geq 0.$$
On the other hand since $K_{X}+L$ is nef and $1$-big we have
\begin{eqnarray*}
&&g_{1}(X,K_{X}+L,K_{X}+L,L)-g_{1}(X,K_{X}+L,K_{X}+L,K_{X}+2L) \\
&&=-\frac{1}{2}(K_{X}+L)^{3}(4K_{X}+5L) \\
&&<0.
\end{eqnarray*}
This is a contradiction.
Therefore $h^{0}(4(K_{X}+L))>0$.
\end{proof}
\noindent
\\
(ii) Next we are going to prove the following:

\begin{Claim} $h^{0}(3(K_{X}+L))\neq 0$ or $h^{0}(5(K_{X}+L))\neq 0$.
\end{Claim}
\begin{proof}
Assume that $h^{0}(3(K_{X}+L))=0$ and $h^{0}(5(K_{X}+L))=0$.
Then the following hold:
\begin{eqnarray*}
0 &\geq& h^{0}(3(K_{X}+L))-h^{0}(2(K_{X}+L)),\\
0 &\leq& h^{0}(4(K_{X}+L))-h^{0}(3(K_{X}+L)),\\
0 &\geq& h^{0}(5(K_{X}+L))-h^{0}(4(K_{X}+L)).
\end{eqnarray*}
By (\ref{Eq}), we have
\begin{eqnarray*}
0&\geq& g_{3}(X,K_{X}+L)+g_{2}(X,K_{X}+L,K_{X}+2L)-h^{2}(\mathcal{O}_{X}),\\
0&\leq& g_{3}(X,K_{X}+L)+g_{2}(X,K_{X}+L,2K_{X}+3L)-h^{2}(\mathcal{O}_{X}),\\
0&\geq& g_{3}(X,K_{X}+L)+g_{2}(X,K_{X}+L,3K_{X}+4L)-h^{2}(\mathcal{O}_{X}).
\end{eqnarray*}

Therefore we get
\begin{eqnarray}
g_{2}(X,K_{X}+L,2K_{X}+3L)\geq g_{2}(X,K_{X}+L,K_{X}+2L), \label{A1}\\
g_{2}(X,K_{X}+L,2K_{X}+3L)\geq g_{2}(X,K_{X}+L,3K_{X}+4L). \label{A2}
\end{eqnarray}

On the other hand by Proposition \ref{B18} we have
\begin{eqnarray*}
&&g_{2}(X,K_{X}+L,2K_{X}+3L)\\
&&=g_{2}(X,K_{X}+L,K_{X}+2L)+g_{2}(X,K_{X}+L,K_{X}+L) \\
&&\ \ \ +g_{1}(X,K_{X}+L,K_{X}+L,K_{X}+2L)-h^{1}(\mathcal{O}_{X})
\end{eqnarray*}
and
\begin{eqnarray*}
&&g_{2}(X,K_{X}+L,3K_{X}+4L)\\
&&=g_{2}(X,K_{X}+L,2K_{X}+3L)+g_{2}(X,K_{X}+L,K_{X}+L) \\
&&\ \ \ +g_{1}(X,K_{X}+L,K_{X}+L,2K_{X}+3L)-h^{1}(\mathcal{O}_{X}).
\end{eqnarray*}

Hence
\begin{eqnarray*}
g_{2}(X,K_{X}+L,K_{X}+L)+g_{1}(X,K_{X}+L,K_{X}+L,K_{X}+2L)-h^{1}(\mathcal{O}_{X})&\geq& 0, \\
g_{2}(X,K_{X}+L,K_{X}+L)+g_{1}(X,K_{X}+L,K_{X}+L,2K_{X}+3L)-h^{1}(\mathcal{O}_{X})&\leq& 0,
\end{eqnarray*}
and therefore we get
$$g_{1}(X,K_{X}+L,K_{X}+L,K_{X}+2L)-g_{1}(X,K_{X}+L,K_{X}+L,2K_{X}+3L)\geq 0.$$

But since $K_{X}+L$ is nef and $1$-big and $3K_{X}+(7/2)L$ is ample, we have
\begin{eqnarray*}
&&g_{1}(X,K_{X}+L,K_{X}+L,K_{X}+2L)-g_{1}(X,K_{X}+L,K_{X}+L,2K_{X}+3L) \\
&&=-(K_{X}+L)^{3}\left(3K_{X}+\frac{7}{2}L\right)\\
&&\leq 0.
\end{eqnarray*}
This is a contradiction.
This completes the proof of this claim.
\end{proof}
\noindent
\\
(ii.1) Next we consider the case of $h^{0}(3(K_{X}+L))>0$.\\
If $h^{0}(3(K_{X}+L))>0$, then by using the positivity of $h^{0}(4(K_{X}+L))$,
we have $h^{0}(m(K_{X}+L))>0$ for every integer $m$ with $m\geq 6$ by Lemmas \ref{Lemma B} and \ref{Lemma C}.
\par
\begin{Claim}
If $h^{0}(3(K_{X}+L))>0$, then $h^{0}(5(K_{X}+L))>0$.
\end{Claim}
\begin{proof}
Assume that $h^{0}(5(K_{X}+L))=0$.
If $h^{0}(2(K_{X}+L))>0$, then by Lemma \ref{Lemma B}
we see that $h^{0}(5(K_{X}+L))>0$.
So we may assume that $h^{0}(2(K_{X}+L))=0$.
Then
\begin{eqnarray*}
0 &\geq& h^{0}(2(K_{X}+L))-h^{0}(K_{X}+L),\\
0 &\leq& h^{0}(3(K_{X}+L))-h^{0}(2(K_{X}+L)),\\
0 &\geq& h^{0}(5(K_{X}+L))-h^{0}(4(K_{X}+L)).
\end{eqnarray*}

By (\ref{Eq}), we have
\begin{eqnarray*}
0&\geq& g_{3}(X,K_{X}+L)+g_{2}(X,K_{X}+L,L)-h^{2}(\mathcal{O}_{X}),\\
0&\leq& g_{3}(X,K_{X}+L)+g_{2}(X,K_{X}+L,K_{X}+2L)-h^{2}(\mathcal{O}_{X}),\\
0&\geq& g_{3}(X,K_{X}+L)+g_{2}(X,K_{X}+L,3K_{X}+4L)-h^{2}(\mathcal{O}_{X}).
\end{eqnarray*}

On the other hand by Proposition \ref{B18} we see that
\begin{eqnarray*}
&&g_{2}(X,K_{X}+L,K_{X}+2L)\\
&&=g_{2}(X,K_{X}+L,L)+g_{2}(X,K_{X}+L,K_{X}+L) \\
&&\ \ \ +g_{1}(X,K_{X}+L,K_{X}+L,L)-h^{1}(\mathcal{O}_{X})
\end{eqnarray*}
and
\begin{eqnarray*}
&&g_{2}(X,K_{X}+L,3K_{X}+4L)\\
&&=g_{2}(X,K_{X}+L,K_{X}+2L)+g_{2}(X,K_{X}+L,2K_{X}+2L) \\
&&\ \ \ +g_{1}(X,K_{X}+L,2(K_{X}+L),K_{X}+2L)-h^{1}(\mathcal{O}_{X})\\
&&=g_{2}(X,K_{X}+L,K_{X}+2L)+2g_{2}(X,K_{X}+L,K_{X}+L) \\
&&\ \ \ +g_{1}(X,K_{X}+L,K_{X}+L,K_{X}+L)
+g_{1}(X,K_{X}+L,2(K_{X}+L),K_{X}+2L)\\
&&\ \ \ -2h^{1}(\mathcal{O}_{X}).
\end{eqnarray*}

Hence
\begin{eqnarray*}
0&\leq& g_{2}(X,K_{X}+L,K_{X}+L)+g_{1}(X,K_{X}+L,K_{X}+L,L)-h^{1}(\mathcal{O}_{X}), \\
0&\geq& 2g_{2}(X,K_{X}+L,K_{X}+L)-2h^{1}(\mathcal{O}_{X})\\
&&+g_{1}(X,K_{X}+L,K_{X}+L,K_{X}+L)+g_{1}(X,K_{X}+L,2(K_{X}+L),K_{X}+2L)
\end{eqnarray*}
and therefore we get
\begin{eqnarray*}
&&2g_{1}(X,K_{X}+L,K_{X}+L,L) \\
&&\geq g_{1}(X,K_{X}+L,K_{X}+L,K_{X}+L)+g_{1}(X,K_{X}+L,2(K_{X}+L),K_{X}+2L).
\end{eqnarray*}

But since $K_{X}+L$ is nef and $1$-big and $7K_{X}+(17/2)L$ is ample, we have
\begin{eqnarray*}
&&2g_{1}(X,K_{X}+L,K_{X}+L,L)-g_{1}(X,K_{X}+L,K_{X}+L,K_{X}+L)\\
&&-g_{1}(X,K_{X}+L,2(K_{X}+L),K_{X}+2L) \\
&&=-(K_{X}+L)^{3}\left(7K_{X}+\frac{17}{2}L\right)\\
&&\leq 0.
\end{eqnarray*}
This is a contradiction.
This completes the proof of this claim.
\end{proof}
\noindent
\\
(ii.2) Next we consider the case of $h^{0}(5(K_{X}+L))>0$.

\begin{Claim}\label{CL5}
If $h^{0}(5(K_{X}+L))>0$, then $h^{0}(m(K_{X}+L))>0$ 
for every integer $m$ with $m\geq 5$.
\end{Claim}
\begin{proof}
If $h^{0}(5(K_{X}+L))>0$, then by using the positivity of $h^{0}(4(K_{X}+L))$,
we have $h^{0}(m(K_{X}+L))>0$ for $m=5, 8, 9, 10$ and $m\geq 12$ by Lemmas \ref{Lemma B} and \ref{Lemma C}.
\par
So we consider the case where $m=6$ (resp. $7$, $11$).
\par
Assume that $h^{0}(6(K_{X}+L))=0$ (resp. $h^{0}(7(K_{X}+L))=0$, $h^{0}(11(K_{X}+L))=0$).
If $h^{0}(3(K_{X}+L))>0$, then by Lemma \ref{Lemma B} (resp. Lemma \ref{Lemma B} and Claim \ref{T-CL-1}, Lemma \ref{Lemma B} and Claim \ref{T-CL-1}) we see that $h^{0}(6(K_{X}+L))>0$ (resp. $h^{0}(7(K_{X}+L))>0$, $h^{0}(11(K_{X}+L))>0$).
So we may assume that $h^{0}(3(K_{X}+L))=0$.
Then
\begin{eqnarray*}
0 &\geq& h^{0}(3(K_{X}+L))-h^{0}(2(K_{X}+L)),\\
0 &\leq& h^{0}(4(K_{X}+L))-h^{0}(3(K_{X}+L)),\\
0 &\geq& h^{0}(6(K_{X}+L))-h^{0}(5(K_{X}+L)) \\
(\mbox{resp.}\ 0 &\geq& h^{0}(7(K_{X}+L))-h^{0}(6(K_{X}+L)), \\
0 &\geq& h^{0}(11(K_{X}+L))-h^{0}(10(K_{X}+L))).
\end{eqnarray*}

By (\ref{Eq}), we have
\begin{eqnarray*}
0&\geq& g_{3}(X,K_{X}+L)+g_{2}(X,K_{X}+L,K_{X}+2L)-h^{2}(\mathcal{O}_{X}),\\
0&\leq& g_{3}(X,K_{X}+L)+g_{2}(X,K_{X}+L,2K_{X}+3L)-h^{2}(\mathcal{O}_{X}),\\
0&\geq& g_{3}(X,K_{X}+L)+g_{2}(X,K_{X}+L,4K_{X}+5L)-h^{2}(\mathcal{O}_{X}) \\
(\mbox{resp.}\ 0&\geq& g_{3}(X,K_{X}+L)+g_{2}(X,K_{X}+L,5K_{X}+6L)-h^{2}(\mathcal{O}_{X}), \\
0&\geq& g_{3}(X,K_{X}+L)+g_{2}(X,K_{X}+L,9K_{X}+10L)-h^{2}(\mathcal{O}_{X})).
\end{eqnarray*}

Hence
\begin{eqnarray*}
g_{2}(X,K_{X}+L,2K_{X}+3L)&\geq&g_{2}(X,K_{X}+L,K_{X}+2L), \\
g_{2}(X,K_{X}+L,2K_{X}+3L)&\geq&g_{2}(X,K_{X}+L,4K_{X}+5L) \\
(\mbox{resp.}\ g_{2}(X,K_{X}+L,2K_{X}+3L)&\geq&g_{2}(X,K_{X}+L,5K_{X}+6L) \\
g_{2}(X,K_{X}+L,2K_{X}+3L)&\geq&g_{2}(X,K_{X}+L,9K_{X}+10L)). 
\end{eqnarray*}

On the other hand by Proposition \ref{B18} we have
\begin{eqnarray*}
&&g_{2}(X,K_{X}+L,2K_{X}+3L)\\
&&=g_{2}(X,K_{X}+L,K_{X}+2L)+g_{2}(X,K_{X}+L,K_{X}+L) \\
&&\ \ \ +g_{1}(X,K_{X}+L,K_{X}+L,K_{X}+2L)-h^{1}(\mathcal{O}_{X}),
\end{eqnarray*}
\begin{eqnarray*}
g_{2}(X,K_{X}+L,4K_{X}+5L)
&=&g_{2}(X,K_{X}+L,2K_{X}+3L)+2g_{2}(X,K_{X}+L,K_{X}+L) \\
&&+g_{1}(X,K_{X}+L,K_{X}+L,K_{X}+L)\\
&&+g_{1}(X,K_{X}+L,2K_{X}+3L,2K_{X}+2L)-2h^{1}(\mathcal{O}_{X}),
\end{eqnarray*}

\begin{eqnarray*}
\mbox{(resp.}
g_{2}(X,K_{X}+L,5K_{X}+6L)
&=&g_{2}(X,K_{X}+L,2K_{X}+3L)+3g_{2}(X,K_{X}+L,K_{X}+L) \\
&&+g_{1}(X,K_{X}+L,K_{X}+L,K_{X}+L)\\
&&+g_{1}(X,K_{X}+L,K_{X}+L,2K_{X}+2L)\\
&&+g_{1}(X,K_{X}+L,2K_{X}+3L,3K_{X}+3L)-3h^{1}(\mathcal{O}_{X})
\end{eqnarray*}
and
\begin{eqnarray*}
g_{2}(X,K_{X}+L,9K_{X}+10L)
&=&g_{2}(X,K_{X}+L,2K_{X}+3L)+7g_{2}(X,K_{X}+L,K_{X}+L) \\
&&+\sum_{k=1}^{6}g_{1}(X,K_{X}+L,K_{X}+L,k(K_{X}+L))\\
&&+g_{1}(X,K_{X}+L,2K_{X}+3L,7K_{X}+7L)-7h^{1}(\mathcal{O}_{X})).
\end{eqnarray*}

Hence we have
\begin{eqnarray*}
&&g_{2}(X,K_{X}+L,K_{X}+L)+g_{1}(X,K_{X}+L,K_{X}+L,K_{X}+2L)-h^{1}(\mathcal{O}_{X})\geq 0,
\end{eqnarray*}
\begin{eqnarray*}
&&2g_{2}(X,K_{X}+L,K_{X}+L)+g_{1}(X,K_{X}+L,K_{X}+L,K_{X}+L)\\
&&+g_{1}(X,K_{X}+L,2K_{X}+3L,2K_{X}+2L)-2h^{1}(\mathcal{O}_{X})\leq 0
\end{eqnarray*}
\begin{eqnarray*}
\mbox{(resp.}
&&3g_{2}(X,K_{X}+L,K_{X}+L)+g_{1}(X,K_{X}+L,K_{X}+L,K_{X}+L)\\
&&+g_{1}(X,K_{X}+L,K_{X}+L,2K_{X}+2L)\\
&&+g_{1}(X,K_{X}+L,2K_{X}+3L,3K_{X}+3L)-3h^{1}(\mathcal{O}_{X})\geq 0,
\end{eqnarray*}
\begin{eqnarray*}
&&7g_{2}(X,K_{X}+L,K_{X}+L)+\sum_{k=1}^{6}g_{1}(X,K_{X}+L,K_{X}+L,k(K_{X}+L))\\
&&+g_{1}(X,K_{X}+L,2K_{X}+3L,7K_{X}+7L)-7h^{1}(\mathcal{O}_{X})\leq 0).
\end{eqnarray*}

Therefore
\begin{eqnarray*}
2g_{1}(X,K_{X}+L,K_{X}+L,K_{X}+2L)
&\geq& g_{1}(X,K_{X}+L,K_{X}+L,K_{X}+L)\\
&&+g_{1}(X,K_{X}+L,2K_{X}+3L,2K_{X}+2L) 
\end{eqnarray*}
\begin{eqnarray*}
(\mbox{resp.}\  
3g_{1}(X,K_{X}+L,K_{X}+L,K_{X}+2L)
&\geq& g_{1}(X,K_{X}+L,2K_{X}+3L,3K_{X}+3L)\\
&&+\sum_{k=1}^{2}g_{1}(X,K_{X}+L,K_{X}+L,k(K_{X}+L)), 
\end{eqnarray*}
\begin{eqnarray*}
7g_{1}(X,K_{X}+L,K_{X}+L,K_{X}+2L)
&\geq& g_{1}(X,K_{X}+L,2K_{X}+3L,7K_{X}+7L)\\
&&+\sum_{k=1}^{6}g_{1}(X,K_{X}+L,K_{X}+L,k(K_{X}+L))).
\end{eqnarray*}

On the other hand, since $K_{X}+L$ is nef and big, we have
\begin{eqnarray*}
&&2g_{1}(X,K_{X}+L,K_{X}+L,K_{X}+2L)\\
&&-g_{1}(X,K_{X}+L,K_{X}+L,K_{X}+L)\\
&&-g_{1}(X,K_{X}+L,2K_{X}+3L,2K_{X}+2L) \\
&&=-(K_{X}+L)^{3}\left(10K_{X}+\frac{23}{2}L\right)\\
&&<0
\end{eqnarray*}

\begin{eqnarray*}
(\mbox{resp.}\ &&3g_{1}(X,K_{X}+L,K_{X}+L,K_{X}+2L)\\
&&-g_{1}(X,K_{X}+L,2K_{X}+3L,3K_{X}+3L)\\
&&-\sum_{k=1}^{2}g_{1}(X,K_{X}+L,K_{X}+L,k(K_{X}+L)) \\
&&=-(K_{X}+L)^{3}\left(22K_{X}+25L\right)\\
&&<0,
\end{eqnarray*}
\begin{eqnarray*}
&&7g_{1}(X,K_{X}+L,K_{X}+L,K_{X}+2L)\\
&&-g_{1}(X,K_{X}+L,2K_{X}+3L,7K_{X}+7L)\\
&&-\sum_{k=1}^{6}g_{1}(X,K_{X}+L,K_{X}+L,k(K_{X}+L)) \\
&&=-(K_{X}+L)^{3}\left(140K_{X}+154L\right)\\
&&<0).
\end{eqnarray*}

This is a contradiction.
Therefore we complete the proof of Claim \ref{CL5}.
\end{proof}
\noindent
\par
Therefore we get the assertion of Theorem \ref{Theorem4.1-big}.
\end{proof}

When we study the positivity of $h^{0}(3(K_{X}+L))$, we need to study the value of the second sectional geometric genus.
Here we fix some notation which will be used in the following results.

\begin{Notation}\label{T-NT1}
Assume that $(X,L)$ is a polarized variety of dimension $4$ which satisfies the assumption $({\rm SRE})$.
Then let $r: X_{1}\to X$ be a resolution of $X$ such that $X_{1}\setminus r^{-1}(\mbox{Sing}(X))\cong X\setminus \mbox{Sing}(X)$ and let $L_{1}=r^{*}(L)$.
\end{Notation}

First we will prove the following proposition. 

\begin{Proposition}\label{T-P1}
Let $(X,L)$ be a polarized variety of dimension $4$ which satisfies the assumption $({\rm SRE})$. We use Notation {\rm \ref{T-NT1}}.
Assume that $K_{X}+L$ is nef and big.
Then for any nef line bundles $A_{1}$ and $A_{2}$ on $X$ the following hold.
\begin{itemize}
\item [\rm (i)] 
$c_{2}(X_{1})r^{*}(A_{1})r^{*}(A_{2})\geq -\frac{1}{8}(18K_{X_{1}}L_{1}+27L_{1}^{2})r^{*}(A_{1})r^{*}(A_{2})$.
\item [\rm (ii)]
One of the following holds.
\begin{itemize}
\item [\rm (ii.1)]
$c_{2}(X_{1})r^{*}(A_{1})r^{*}(A_{2})\geq -\frac{1}{3}(6K_{X_{1}}L_{1}+8L_{1}^{2})r^{*}(A_{1})r^{*}(A_{2})$.
\item [\rm (ii.2)]
$X$ is rationally connected and $h^{0}(K_{X}+2L)=h^{0}(K_{X}+L)=0$.
\end{itemize}
\end{itemize}
\end{Proposition}
\begin{proof}
(1) First we assume that $\Omega_{X}\left\langle \frac{3}{4}L\right\rangle$ is not generically nef. (Here $\Omega_{X}\left\langle \frac{3}{4}L\right\rangle$ denotes $\mathbb{Q}$-twisted sheaf. See \cite[2.3 Definition]{Horing09}.)
Then by \cite[3.1 Theorem]{Horing09}, there exist a smooth projective variety $X^{\prime}$ of dimension $4$, a smooth projective variety $Y$ of dimension $m$ with $m\leq 3$,
a birational morphism $\mu : X^{\prime}\to X$, and a surjective morphism $\varphi: X^{\prime}\to Y$ such that the following holds:
The general fiber $F$ of $\varphi$ is rationally connected and $h^{0}(D)=0$ for any Cartier divisor $D$ on $F$ such that $D\sim_{\mathbb{Q}} K_{F}+j\mu^{*}(\frac{3}{4}\mu^{*}(L))_{F}$ for any $j\in [0,4-m]\cap \mathbb{Q}$.
(Here $\sim_{\mathbb{Q}}$ denotes the linear equivalence of $\mathbb{Q}$-divisors.)
\\
(1.0) Assume that $\dim Y=0$.
Then $X^{\prime}$ is rationally connected and $h^{0}(K_{X^{\prime}}+3\mu^{*}(L))=h^{0}(K_{X^{\prime}}+2\mu^{*}(L))=h^{0}(K_{X^{\prime}}+\mu^{*}(L))=0$.
Here we note that $\chi(\mathcal{O}_{X^{\prime}})=1$ since $h^{i}(\mathcal{O}_{X^{\prime}})=0$ for any $i\geq 1$.
But by \cite[4.1 Lemma]{Horing09} this is impossible because 
$$(K_{X^{\prime}}+3\mu^{*}(L))\mu^{*}(L)^{3}=(K_{X}+3L)L^{3}>0.$$
\noindent
\\
(1.1) Assume that $\dim Y=1$.
Then for the general fiber $F$ of $\varphi$ we have $h^{0}(K_{F}+2\mu^{*}(L)|_{F})=h^{0}(K_{F}+\mu^{*}(L)|_{F})=0$.
But $\kappa(K_{F}+2\mu^{*}(L)|_{F})\geq 0$ holds because $K_{X}+L$ is nef.
Hence $h^{0}(K_{F}+2\mu^{*}(L)|_{F})>0$ by \cite[Theorem 4.6]{Fukuma10-2} since $\dim F=3$.
This is impossible.
\\
(1.2) Assume that $\dim Y=2$.
Then for the general fiber $F$ of $\varphi$ we have $h^{0}(K_{F}+\mu^{*}(L)|_{F})=0$.
On the other hand we have $\kappa(K_{F}+\mu^{*}(L)|_{F})\geq 0$ because $K_{X}+L$ is nef.
Hence $h^{0}(K_{F}+\mu^{*}(L)|_{F})>0$ by \cite[Proposition 2.1]{Fukuma10-2} since $\dim F=2$.
This is also impossible.
\\
(1.3) Assume that $\dim Y=3$.
In this case $F\cong \mathbb{P}^{1}$.
If $\deg \mu^{*}(L)|_{F}\geq 3$, then there exists $j\in [0,1]\cap \mathbb{Q}$ such that 
$K_{F}+j\mu^{*}(\frac{3}{4}\mu^{*}(L))_{F}$ is a Cartier divisor with $\deg K_{F}+j\mu^{*}(\frac{3}{4}\mu^{*}(L))_{F}\geq 0$.
Hence $h^{0}(K_{F}+j\mu^{*}(\frac{3}{4}\mu^{*}(L))_{F})>0$ and this is a contradiction.
Therefore $\deg \mu^{*}(L)|_{F}\leq 2$.
In particular 
\begin{equation}
\deg (K_{F}+\mu^{*}(L)|_{F})\leq 0. \label{EQ-Ta}
\end{equation}
On the other hand we have
\begin{eqnarray}
K_{F}+\mu^{*}(L)_{F}
&=&(K_{X^{\prime}}+\mu^{*}(L))_{F} \label{EQ-Tb}\\
&=&(\mu^{*}(K_{X}+L)+E_{\mu})_{F},\nonumber
\end{eqnarray}
where $E_{\mu}$ is an effective $\mu$-exceptional divisor.
Since $K_{X}+L$ is nef and big, we see that $\mu^{*}(K_{X}+L)$ is also nef and big.
Hence $(\mu^{*}(K_{X}+L))_{F}$ is nef and big (see \cite[(1.4) Proposition]{Fujita89}).
So we get 
\begin{equation}
\deg \mu^{*}(K_{X}+L)_{F}>0. \label{EQ-Tc}
\end{equation}
Here we note that we can take a general fiber $F$ of $\varphi$ with $F\not\subset\mbox{Supp}(E_{\mu})$.
Therefore 
\begin{equation}
\deg (E_{\mu})_{F}\geq 0. \label{EQ-Td}
\end{equation}
By (\ref{EQ-Tb}), (\ref{EQ-Tc}) and (\ref{EQ-Td}) we have $\deg (K_{F}+\mu^{*}(L)|_{F})>0$
and this contradicts to (\ref{EQ-Ta}).
\par
By (1.0), (1.1), (1.2) and (1.3) we conclude that $\Omega_{X}\left\langle \frac{3}{4}L\right\rangle$ is generically nef.
Here we note that $K_{X}+3L$ is nef.
Hence by \cite[2.10 Lemma]{Horing09} we get
$$c_{2}\left(\Omega_{X}\left\langle \frac{3}{4}L\right\rangle\right)A_{1}A_{2}\geq 0.$$
Namely we have
$$c_{2}(X)A_{1}A_{2}\geq -\frac{1}{8}(18K_{X}L+27L^{2})A_{1}A_{2}.$$
Here we note that $X$ has only isolated singularities and $\dim X=4$.
Hence we have $c_{2}(X)A_{1}A_{2}=c_{2}(X_{1})r^{*}(A_{1})r^{*}(A_{2})$.
On the other hand, we have $(18K_{X_{1}}L_{1}+27L_{1}^{2})r^{*}(A_{1})r^{*}(A_{2})=(18K_{X}L+27L^{2})A_{1}A_{2}$.
Therefore we get the assertion of (i).
\\
\\
(2.1) First we assume that $\Omega_{X}\left\langle \frac{2}{3}L\right\rangle$ is generically nef.
Here we note that $K_{X}+\frac{8}{3}L$ is nef.
Hence by \cite[2.10 Lemma]{Horing09} we get
$$c_{2}\left(\Omega_{X}\left\langle \frac{2}{3}L\right\rangle\right)A_{1}A_{2}\geq 0.$$
Namely we have
$$c_{2}(X)A_{1}A_{2}\geq -\frac{1}{3}(6K_{X}L+8L^{2})A_{1}A_{2}.$$
Here we note that $X$ has only isolated singularities and $\dim X=4$.
Hence we have $c_{2}(X)A_{1}A_{2}=c_{2}(X_{1})r^{*}(A_{1})r^{*}(A_{2})$.
On the other hand, we have $(6K_{X_{1}}L_{1}+8L_{1}^{2})r^{*}(A_{1})r^{*}(A_{2})=(6K_{X}L+8L^{2})A_{1}A_{2}$.
Therefore we get (ii.1) in the statement of Proposition \ref{T-P1}.
\par
Next we assume that $\Omega_{X}\left\langle \frac{2}{3}L\right\rangle$ is not generically nef.
Then by \cite[3.1 Theorem]{Horing09}, there exist a smooth projective variety $X^{\prime}$ of dimension $4$, a smooth projective variety $Y$ of dimension $m$ with $m\leq 3$,
a birational morphism $\mu : X^{\prime}\to X$, and a surjective morphism $\varphi: X^{\prime}\to Y$ such that the following holds:
The general fiber $F$ of $\varphi$ is rationally connected and $h^{0}(D)=0$ for any Cartier divisor $D$ on $F$ such that $D\sim_{\mathbb{Q}} K_{F}+j\mu^{*}(\frac{2}{3}\mu^{*}(L))_{F}$ for any $j\in [0,4-m]\cap \mathbb{Q}$.
\par
By the same argument as in the cases (1.1), (1.2) and (1.3) above, we can prove $m=0$.
Then $X^{\prime}$ is rationally connected and $h^{0}(K_{X^{\prime}}+2\mu^{*}(L))=h^{0}(K_{X^{\prime}}+\mu^{*}(L))=0$.
Therefore we get (ii.2).
\end{proof} 

\begin{Theorem}\label{T-TH1}
Let $(X,L)$ be a polarized variety of dimension $4$ which satisfies the assumption $({\rm SRE})$.
Assume that $K_{X}+L$ is nef and big.
Then $h^{0}(3(K_{X}+L))>0$.
\end{Theorem}
\begin{proof}
Assume that $h^{0}(3(K_{X}+L))=0$.
Then by Lemma \ref{Lemma B}, we get $h^{0}(K_{X}+L)=0$.
Therefore
\begin{eqnarray*}
0&\leq& h^{0}(2(K_{X}+L))-h^{0}(K_{X}+L) \\
0&\geq& h^{0}(3(K_{X}+L))-h^{0}(2(K_{X}+L)).
\end{eqnarray*}
By using (\ref{Eq}) in the proof of Claim \ref{T-CL-1} we have
\begin{eqnarray}
0&\leq& g_{3}(X,K_{X}+L)+g_{2}(X,L,K_{X}+L)-h^{2}(\mathcal{O}_{X})
\label{T-Eq3-1}\\
0&\geq& g_{3}(X,K_{X}+L)+g_{2}(X,K_{X}+L,K_{X}+2L)-h^{2}(\mathcal{O}_{X}).
\label{T-Eq3-2}
\end{eqnarray}
By (\ref{T-Eq3-1}) and (\ref{T-Eq3-2}) we get
$$g_{2}(X,K_{X}+L,K_{X}+2L)\leq g_{2}(X,L,K_{X}+L).$$
On the other hand by Proposition \ref{B18}
\begin{eqnarray*}
g_{2}(X,K_{X}+L,K_{X}+2L)
&=&g_{2}(X,K_{X}+L,K_{X}+L)+g_{2}(X,K_{X}+L,L)\\
&&+g_{1}(X,K_{X}+L,K_{X}+L,L)-h^{1}(\mathcal{O}_{X}).
\end{eqnarray*}

Hence we get 
\begin{equation}
g_{2}(X,K_{X}+L,K_{X}+L)+g_{1}(X,K_{X}+L,K_{X}+L,L)-h^{1}(\mathcal{O}_{X})\leq 0.
\label{T-SG-0}
\end{equation}
\noindent
\\
(1) First we assume that $(X,L)$ satisfies (ii.1) in Proposition \ref{T-P1}.
We use Notation \ref{T-NT1}.
Then by \cite[(2.2.A)]{Fukuma04-4} and \cite[Lemma 3.1]{Fukuma10-2} we have
\begin{eqnarray*}
&&g_{2}(X,K_{X}+L,K_{X}+L) \\
&&=g_{2}(X_{1},r^{*}(K_{X}+L),r^{*}(K_{X}+L))\\
&&=-1+h^{1}(\mathcal{O}_{X_{1}})+\frac{1}{12}(K_{X_{1}}+3r^{*}(K_{X}+L))(K_{X_{1}}+2r^{*}(K_{X}+L))r^{*}(K_{X}+L)^{2}\\
&&\ \ \ +\frac{1}{12}c_{2}(X_{1})r^{*}(K_{X}+L)^{2}+\frac{1}{24}(2K_{X_{1}}+2r^{*}(K_{X}+L))r^{*}(K_{X}+L)^{3} \\
&&\geq -1+h^{1}(\mathcal{O}_{X_{1}})+\frac{1}{12}(K_{X_{1}}+3r^{*}(K_{X}+L))(K_{X_{1}}+2r^{*}(K_{X}+L))r^{*}(K_{X}+L)^{2}\\
&&\ \ \ -\frac{1}{36}(6K_{X_{1}}L_{1}+8L_{1}^{2})r^{*}(K_{X}+L)^{2}+\frac{1}{24}(2K_{X_{1}}+2r^{*}(K_{X}+L))r^{*}(K_{X}+L)^{3} \\
&&=-1+h^{1}(\mathcal{O}_{X_{1}})+\frac{7}{6}(K_{X}+L)^{4}-\frac{5}{6}(K_{X}+L)^{3}L+\frac{1}{36}(K_{X}+L)^{2}L^{2}.
\end{eqnarray*}

On the other hand we have
\begin{equation}
g_{1}(X,K_{X}+L,K_{X}+L,L)=1+\frac{3}{2}(K_{X}+L)^{3}L. \label{SG-1}
\end{equation}
Hence
\begin{eqnarray*}
&&g_{2}(X,K_{X}+L,K_{X}+L)+g_{1}(X,K_{X}+L,K_{X}+L,L)-h^{1}(\mathcal{O}_{X})\\
&&\geq -1+h^{1}(\mathcal{O}_{X_{1}})+\frac{7}{6}(K_{X}+L)^{4}-\frac{5}{6}(K_{X}+L)^{3}L\\
&&\ \ \ +\frac{1}{36}(K_{X}+L)^{2}L^{2}+1+\frac{3}{2}(K_{X}+L)^{3}L-h^{1}(\mathcal{O}_{X_{1}}) \\
&&=\frac{7}{6}(K_{X}+L)^{4}+\frac{2}{3}(K_{X}+L)^{3}L+\frac{1}{36}(K_{X}+L)^{2}L^{2}\\
&&>0.
\end{eqnarray*}

This contradicts (\ref{T-SG-0}).
Therefore this case cannot occur.
\\
\\
(2) Next we assume that $(X,L)$ satisfies (ii.2) in Proposition \ref{T-P1}.
By using Proposition \ref{T-P1} (i), \cite[(2.2.A)]{Fukuma04-4} and \cite[Lemma 3.1]{Fukuma10-2}, we have
\begin{eqnarray*}
&&g_{2}(X,K_{X}+L,K_{X}+L) \\
&&=g_{2}(X_{1},r^{*}(K_{X}+L),r^{*}(K_{X}+L))\\
&&\geq -1+h^{1}(\mathcal{O}_{X_{1}})+\frac{1}{12}(K_{X_{1}}+3r^{*}(K_{X}+L))(K_{X_{1}}+2r^{*}(K_{X}+L))r^{*}(K_{X}+L)^{2}\\
&&\ \ \ -\frac{1}{32}(6K_{X_{1}}L_{1}+9L_{1}^{2})r^{*}(K_{X}+L)^{2}+\frac{1}{24}(2K_{X_{1}}+2r^{*}(K_{X}+L))r^{*}(K_{X}+L)^{3} \\
&&=-1+h^{1}(\mathcal{O}_{X_{1}})+\frac{7}{6}(K_{X}+L)^{4}-\frac{41}{48}(K_{X}+L)^{3}L-\frac{1}{96}(K_{X}+L)^{2}L^{2}.
\end{eqnarray*}

Hence by (\ref{SG-1}) we have
\begin{eqnarray*}
&&g_{2}(X,K_{X}+L,K_{X}+L)+g_{1}(X,K_{X}+L,K_{X}+L,L)-h^{1}(\mathcal{O}_{X})\\
&&\geq \frac{7}{6}(K_{X}+L)^{4}+\frac{31}{48}(K_{X}+L)^{3}L-\frac{1}{96}(K_{X}+L)^{2}L^{2}.
\end{eqnarray*}

Here we note that by \cite[Theorem 3.2 (i)]{Fukuma10-2} we have
\begin{equation}
g_{2}(X,K_{X}+L,K_{X}+L)=h^{0}(3K_{X}+2L)-2h^{0}(2K_{X}+L) \label{SSG-1}
\end{equation} 
because $X$ is rationally connected.
If $g_{2}(X,K_{X}+L,K_{X}+L)\geq 0$, then $g_{2}(X,K_{X}+L,K_{X}+L)+g_{1}(X,K_{X}+L,K_{X}+L,L)-h^{1}(\mathcal{O}_{X})\geq g_{1}(X,K_{X}+L,K_{X}+L,L)>0$ because $h^{1}(\mathcal{O}_{X})=0$ and this contradicts (\ref{T-SG-0}).
Therefore $g_{2}(X,K_{X}+L,K_{X}+L)<0$ 
and we see from (\ref{SSG-1}) that $h^{0}(2K_{X}+L)>0$.
Here we note that
\begin{eqnarray*}
(K_{X}+L)^{3}L&=&(K_{X}+\frac{1}{2}L)(K_{X}+L)^{2}L+\frac{1}{2}(K_{X}+L)^{2}L^{2}.
\end{eqnarray*}

Then
\begin{eqnarray*}
\frac{31}{48}(K_{X}+L)^{3}L
&=&\frac{31}{48}(K_{X}+\frac{1}{2}L)(K_{X}+L)^{2}L+\frac{31}{96}(K_{X}+L)^{2}L^{2}\\
&\geq& \frac{31}{96}(K_{X}+L)^{2}L^{2}.
\end{eqnarray*}

Therefore
\begin{eqnarray*}
&&g_{2}(X,K_{X}+L,K_{X}+L)+g_{1}(X,K_{X}+L,K_{X}+L,L)-h^{1}(\mathcal{O}_{X})\\
&&\geq \frac{7}{6}(K_{X}+L)^{4}+\frac{31}{48}(K_{X}+L)^{3}L-\frac{1}{96}(K_{X}+L)^{2}L^{2}\\
&&\geq \frac{7}{6}(K_{X}+L)^{4}+\frac{5}{16}(K_{X}+L)^{2}L^{2}\\
&&>0.
\end{eqnarray*}

This is also impossible.
\par
Therefore we get $h^{0}(3(K_{X}+L))>0$.
\end{proof}

By Theorems \ref{Theorem4.1-big}, \ref{T-TH1} and Remark \ref{TREM-1} 
we get the following corollary.

\begin{Corollary}\label{Corollary4.1}
$m_{4}^{{\rm NEF}}({\rm SRE})\leq m_{4}^{{\rm NEF}}({\rm SRE})^{+}\leq 3$.
\end{Corollary}

Next we consider the case of $\kappa(X)\geq 0$.

\begin{Theorem}\label{Theorem4.1.5}
Let $(X,L)$ be a polarized variety of dimension $4$.
Assume that $(X,L)$ satisfies the assumption $({\rm SRE})$,
$\kappa(X)\geq 0$ and $K_{X}+L$ is nef.
Then $h^{0}(2(K_{X}+L))\geq 1$.
\end{Theorem}
\begin{proof}
We use Notation \ref{T-NT1}.
Then $K_{X_{1}}+L_{1}=r^{*}(K_{X}+L)+E_{r}$ holds, where $E_{r}$ is an effective $r$-exceptional divisor.
Hence for any positive integer $m$
\begin{eqnarray*}
h^{0}(m(K_{X_{1}}+L_{1}))
&=&h^{0}(mr^{*}(K_{X}+L)+mE_{r})\\
&=&h^{0}(mr^{*}(K_{X}+L))\\
&=&h^{0}(m(K_{X}+L)).
\end{eqnarray*}

Here we also note that
\begin{eqnarray*}
h^{0}(2K_{X_{1}}+2L_{1})
&=&h^{0}(K_{X_{1}}+K_{X_{1}}+L_{1}+L_{1}) \\
&=&h^{0}(K_{X_{1}}+r^{*}(K_{X}+L)+E_{r}+L_{1}) \\
&\geq &h^{0}(K_{X_{1}}+r^{*}(K_{X}+L)+L_{1}) \\
&=&h^{0}(r^{*}(K_{X})+r^{*}(K_{X}+L)+r^{*}(L)+E_{r}) \\
&\geq &h^{0}(r^{*}(2K_{X}+2L)) \\
&=&h^{0}(2K_{X}+2L).
\end{eqnarray*}

Since $h^{0}(2K_{X_{1}}+2L_{1})=h^{0}(2K_{X}+2L)$, we have
$h^{0}(2K_{X_{1}}+2L_{1})=h^{0}(K_{X_{1}}+r^{*}(K_{X}+L)+L_{1})$.

Then 
\begin{eqnarray}
&&h^{0}(2(K_{X}+L))-h^{0}(K_{X}+L) \label{4.1.5.0}\\
&&=h^{0}(2(K_{X_{1}}+L_{1}))-h^{0}(K_{X_{1}}+L_{1}) \nonumber \\
&&=h^{0}(K_{X_{1}}+r^{*}(K_{X}+L)+L_{1})-h^{0}(K_{X_{1}}+L_{1}) \nonumber \\
&&=g_{3}(X_{1},r^{*}(K_{X}+L))+g_{2}(X_{1},r^{*}(K_{X}+L),L_{1})-h^{2}(\mathcal{O}_{X_{1}})\nonumber\\
&&=-\chi_{3}^{H}(X_{1},r^{*}(K_{X}+L))+\chi_{2}^{H}(X_{1},r^{*}(K_{X}+L),L_{1}).\nonumber
\end{eqnarray}
By using \cite[(2.2.B)]{Fukuma04-4} and \cite[Corollary 2.7]{Fukuma08} we have
\begin{eqnarray*}
&&-\chi_{3}^{H}(X_{1},r^{*}(K_{X}+L))\\
&&=\frac{1}{24}(r^{*}(K_{X}+L))^{4}+\frac{1}{12}K_{X_{1}}(r^{*}(K_{X}+L))^{3}\\
&&\ \ \ +\frac{1}{24}(K_{X_{1}}^{2}+c_{2}(X_{1}))(r^{*}(K_{X}+L))^{2}
+\frac{1}{24}K_{X_{1}}c_{2}(X_{1})(r^{*}(K_{X}+L)),
\end{eqnarray*}
and
\begin{eqnarray*}
&&\chi_{2}^{H}(X_{1},r^{*}(K_{X}+L),L_{1})\\
&&=\frac{1}{6}L_{1}(r^{*}(K_{X}+L))^{3}+\frac{1}{4}(L_{1})^{2}(r^{*}(K_{X}+L))^{2}
+\frac{1}{6}(L_{1})^{3}(r^{*}(K_{X}+L))\\
&&\ \ \ +\frac{1}{4}K_{X_{1}}(r^{*}(K_{X}+L)+L_{1})(r^{*}(K_{X}+L))L_{1}
+\frac{1}{12}(K_{X_{1}}^{2}+c_{2}(X_{1}))(r^{*}(K_{X}+L))L_{1}.
\end{eqnarray*}

Hence
\begin{eqnarray}
&&-{\chi}_{3}^{H}({X_{1}},r^{*}(K_{X}+L))+{\chi}_{2}^{H}({X_{1}},r^{*}(K_{X}+L),L_{1}) \nonumber\\
&&=\frac{1}{24}(r^{*}(K_{X}+L))^{2}
\left\{ (r^{*}(K_{X}+L))^{2}+2K_{X_{1}}r^{*}(K_{X}+L)+K_{X_{1}}^{2}+4L_{1}r^{*}(K_{X}+L)\right.\nonumber\\
&&\ \ \ \left.+6L_{1}^{2}+6K_{X_{1}}L_{1}\right\}+\frac{1}{24}c_{2}(X_{1})(r^{*}(K_{X}+L)^{2}+K_{X_{1}}r^{*}(K_{X}+L)+2L_{1}r^{*}(K_{X}+L))
\nonumber\\
&&\ \ \ +\frac{1}{24}(4L_{1}^{3}r^{*}(K_{X}+L)+6K_{X_{1}}L_{1}^{2}r^{*}(K_{X}+L)+2K_{X_{1}}^{2}L_{1}r^{*}(K_{X}+L))\nonumber\\
&&=\frac{1}{24}(r^{*}(K_{X}+L))^{2}(K_{X_{1}}+r^{*}(K_{X}+L))^{2}
+\frac{5}{12}(r^{*}(K_{X}+L))^{3}L_{1}\nonumber\\
&&\ \ \ +\frac{1}{24}c_{2}(X_{1})r^{*}(K_{X}+L)(r^{*}(K_{X}+L)+K_{X_{1}}+2L_{1})
+\frac{1}{12}r^{*}(K_{X}+L)^{2}r^{*}(K_{X}+2L)L_{1}. \nonumber
\end{eqnarray}

Here we note that by Lemma \ref{Lemma E} we have $c_{2}(X_{1})E_{r}r^{*}(K_{X}+L)=0$.
We also note that by Lemma \ref{Lemma D}
$(r^{*}(K_{X}+L))^{2}(K_{X_{1}}+r^{*}(K_{X}+L))^{2}=(r^{*}(K_{X}+L))^{2}(r^{*}(2K_{X}+L))^{2}$
holds.

Hence
\begin{eqnarray}
&&-\chi_{3}^{H}(X_{1},r^{*}(K_{X}+L))+\chi_{2}^{H}(X_{1},r^{*}(K_{X}+L),L_{1}) \label{4.1.5.1}\\
&&=\frac{1}{24}(r^{*}(K_{X}+L))^{2}(r^{*}(2K_{X}+L))^{2}
+\frac{5}{12}(r^{*}(K_{X}+L))^{3}r^{*}(L) \nonumber\\
&&+\frac{1}{24}c_{2}(X_{1})r^{*}(K_{X}+L)(r^{*}(2K_{X}+3L))
+\frac{1}{12}r^{*}(K_{X}+L)^{2}r^{*}(K_{X}+2L)r^{*}(L). \nonumber
\end{eqnarray}

By setting $H_{1}:=2K_{X}+3L$, $H_{2}:=K_{X}+L$, $H:=L$ and $s=1$,
and by applying \cite[Theorem 4.2]{Fukuma10-2}, we have
\begin{eqnarray}
&&c_{2}(X_{1})r^{*}(2K_{X}+3L)r^{*}(K_{X}+L) \label{4.1.5.2}\\
&&\geq -\frac{3}{4}K_{X_{1}}r^{*}(L)r^{*}(K_{X}+L)r^{*}(2K_{X}+3L)
-\frac{3}{8}(r^{*}(L))^{2}r^{*}(K_{X}+L)r^{*}(2K_{X}+3L).\nonumber\\
&&=-\frac{3}{4}r^{*}(K_{X})r^{*}(L)r^{*}(K_{X}+L)r^{*}(2K_{X}+3L)
-\frac{3}{8}(r^{*}(L))^{2}r^{*}(K_{X}+L)r^{*}(2K_{X}+3L).\nonumber
\end{eqnarray}

By (\ref{4.1.5.0}), (\ref{4.1.5.1}) and (\ref{4.1.5.2}), we get
\begin{eqnarray}
&&h^{0}(2(K_{X}+L))-h^{0}(K_{X}+L) \label{4.1.5.3} \\
&&=-\chi_{3}^{H}(X_{1},r^{*}(K_{X}+L))+\chi_{2}^{H}(X_{1},r^{*}(K_{X}+L),L_{1})\nonumber\\
&&\geq \frac{1}{24}(K_{X}+L)^{2}(2K_{X}+L)^{2}+\frac{5}{12}(K_{X}+L)^{3}L \nonumber\\
&&\ \ \ -\frac{1}{24}\left(\frac{3}{4}K_{X}L(K_{X}+L)(2K_{X}+3L)+\frac{3}{8}L^{2}(K_{X}+L)(2K_{X}+3L)\right)\nonumber\\
&&\ \ \ +\frac{1}{12}(K_{X}+L)^{2}(K_{X}+2L)L\nonumber\\
&&=\frac{1}{192}(K_{X}+L)\left\{
8(K_{X}+L)(2K_{X}+L)^{2}+80(K_{X}+L)^{2}L \right.\nonumber\\
&&\ \ \ \left.-6K_{X}L(2K_{X}+3L)-3L^{2}(2K_{X}+3L)+16(K_{X}+L)(K_{X}+2L)L\right\}\nonumber\\
&&=\frac{1}{192}(K_{X}+L)\left\{
32K_{X}(K_{X}+2L)^{2}+20K_{X}(K_{X}+2L)L+56(K_{X}+L)L^{2}+55L^{3}\right\}\nonumber\\
&&\geq \frac{1}{192}(K_{X}+L)\left\{
56(K_{X}+L)L^{2}+55L^{3}\right\}\nonumber\\
&&\geq \frac{111}{192}. \nonumber
\end{eqnarray}

Therefore we get the assertion.
\end{proof}

\begin{Theorem}\label{Theorem4.2}
Let $(X,L)$ be a polarized variety of dimension $4$.
Assume that $(X,L)$ satisfies the assumption $(\mbox{\rm SRE})$, 
$K_{X}+L$ is nef and $\kappa(X)\geq 0$.
Then for every integer $m$ with $m\geq 2$
$$h^{0}(m(K_{X}+L))\geq \frac{(m-1)(m-2)(m^{2}+3m+6)}{12}+1.$$
\end{Theorem}
\begin{proof}
We use Notation \ref{T-NT1}.
As in the proof of Theorem \ref{Theorem4.1.5}, we have
$h^{0}(m(K_{X_{1}}+L_{1}))=h^{0}(m(K_{X}+L))$ for any positive inetger $m$.
On the other hand,
\begin{eqnarray*}
h^{0}(m(K_{X_{1}}+L_{1}))
&=&h^{0}(K_{X_{1}}+(m-1)(K_{X_{1}}+L_{1})+L_{1})\\
&=&h^{0}(K_{X_{1}}+(m-1)r^{*}(K_{X}+L)+r^{*}(L))\\
&=&h^{0}(K_{X_{1}}+(m-2)r^{*}(K_{X})+(m-1)r^{*}(L)+r^{*}(K_{X}+L)),\\
h^{0}((m-1)(K_{X_{1}}+L_{1}))
&=&h^{0}(K_{X_{1}}+(m-2)r^{*}(K_{X})+(m-1)r^{*}(L)).
\end{eqnarray*}

Let
$F(t):=h^{0}(t(K_{X}+L))-h^{0}((t-1)(K_{X}+L))$.
Then by Theorem \ref{I1} we have
\begin{eqnarray*}
F(t)&=&
h^{0}(K_{X_{1}}+(t-2)r^{*}(K_{X})+(t-1)r^{*}(L)+r^{*}(K_{X}+L))\\
&&-h^{0}(K_{X_{1}}+(t-2)r^{*}(K_{X})+(t-1)r^{*}(L))\\
&=&
g_{3}(X_{1},r^{*}(K_{X}+L))+g_{2}(X_{1},(t-2)r^{*}(K_{X}+L)+r^{*}(L),r^{*}(K_{X}+L))
-h^{2}(\mathcal{O}_{X}).
\end{eqnarray*}
Hence we have
\begin{eqnarray*}
F(t)-F(t-1)
&=&g_{2}(X_{1},(t-2)r^{*}(K_{X}+L)+r^{*}(L),r^{*}(K_{X}+L))\\
&&-g_{2}(X_{1},(t-3)r^{*}(K_{X}+L)+r^{*}(L),r^{*}(K_{X}+L)).
\end{eqnarray*}
On the other hand by Proposition \ref{B18}
\begin{eqnarray*}
&&g_{2}(X_{1},(t-2)r^{*}(K_{X}+L)+r^{*}(L),r^{*}(K_{X}+L)) \\
&&=g_{2}(X_{1},(t-3)r^{*}(K_{X}+L)+r^{*}(L),r^{*}(K_{X}+L))+g_{2}(X_{1},r^{*}(K_{X}+L),r^{*}(K_{X}+L))\\
&&\ \ \ +g_{1}(X_{1},(t-3)r^{*}(K_{X}+L)+r^{*}(L),r^{*}(K_{X}+L),r^{*}(K_{X}+L))-h^{1}(\mathcal{O}_{X_{1}}).
\end{eqnarray*}

Hence
\begin{eqnarray*}
F(t)-F(t-1)
&=&g_{1}(X_{1},(t-3)r^{*}(K_{X}+L)+r^{*}(L),r^{*}(K_{X}+L),r^{*}(K_{X}+L))\\
&&+g_{2}(X_{1},r^{*}(K_{X}+L),r^{*}(K_{X}+L))-h^{1}(\mathcal{O}_{X_{1}}).
\end{eqnarray*}
Here we note that $(K_{X}+L)L^{3}\geq 2$ because $(K_{X}+L)L^{3}$ is positive and even.
Moreover since
$$(K_{X}+L)^{4}\geq (K_{X}+L)^{3}L\geq (K_{X}+L)^{2}L^{2}\geq (K_{X}+L)L^{3}
\geq 2,$$
we have
\begin{eqnarray*}
&&g_{1}(X_{1},(t-3)r^{*}(K_{X}+L)+r^{*}(L),r^{*}(K_{X}+L),r^{*}(K_{X}+L)) \\
&&=1+\frac{t}{2}((t-3)K_{X}+(t-2)L)(K_{X}+L)^{3} \\
&&=1+\frac{t(t-3)}{2}(K_{X}+L)^{4}+\frac{t}{2}L(K_{X}+L)^{3} \\
&&\geq 1+t(t-3)+t\\
&&=(t-1)^{2}.
\end{eqnarray*}
By \cite[Corollary 4.1]{Fukuma10-2} we have $g_{2}(X_{1},r^{*}(K_{X}+L),r^{*}(K_{X}+L))\geq h^{1}(\mathcal{O}_{X_{1}})$.
Hence
$$F(t)-F(t-1)\geq (t-1)^{2}$$
and
\begin{eqnarray*}
&&h^{0}(k(K_{X}+L))-h^{0}((k-1)(K_{X}+L))\\
&&=F(k) \\
&&\geq (k-1)^{2}+\cdots+2^{2}+F(2)\\
&&=\frac{k(k-1)(2k-1)}{6}-1+F(2).
\end{eqnarray*}

\begin{Claim}\label{Claim4.4}
$h^{0}(2(K_{X}+L))-h^{0}(K_{X}+L)\geq 0$.
\end{Claim}
\begin{proof}
If $h^{0}(K_{X}+L)\geq 1$, then by Lemma \ref{Lemma B} 
we get $h^{0}(2(K_{X}+L))-h^{0}(K_{X}+L)\geq h^{0}(K_{X}+L)-1\geq 0$.
Hence $h^{0}(2(K_{X}+L))-h^{0}(K_{X}+L)\geq 0$.
If $h^{0}(K_{X}+L)=0$, then by Theorem \ref{Theorem4.1.5} we get 
$h^{0}(2(K_{X}+L))-h^{0}(K_{X}+L)\geq 1$.
So we get the assertion.
\end{proof}
\noindent
\par
By Claim \ref{Claim4.4}, $F(2)\geq 0$.
Hence
$$h^{0}(k(K_{X}+L))-h^{0}((k-1)(K_{X}+L))\geq \frac{k(k-1)(2k-1)}{6}-1.$$
Therefore
\begin{eqnarray*}
h^{0}(m(K_{X}+L))
&\geq& h^{0}(2(K_{X}+L))+\sum_{k=3}^{m}\left\{\frac{k(k-1)(2k-1)}{6}-1\right\} \\
&=&h^{0}(2(K_{X}+L))+\frac{1}{12}m^{2}(m^{2}-1)-(m-1)\\
&\geq&1+\frac{1}{12}m^{2}(m^{2}-1)-(m-1)\\
&=&\frac{(m-1)(m-2)(m^{2}+3m+6)}{12}+1.
\end{eqnarray*}
We get the assertion of Theorem \ref{Theorem4.2}.
\end{proof}

\section{The case of \textrm{\boldmath $\kappa(K_{X}+L)\geq 0$}}

In this section, we consider the case of $\kappa(K_{X}+L)\geq 0$ in general.
First we review the adjunction theory of Beltrametti-Sommese and Fujita, which will be used later.

\begin{Theorem}\label{TH1}
Let $(X,\mathcal{L})$ be a polarized manifold with $\dim X=n\geq 3$.
Then $(X,\mathcal{L})$ is one of the following types.
\begin{itemize}
\item [\rm (1)]
$(\mathbb{P}^{n}, \mathcal{O}_{\mathbb{P}^{n}}(1))$.
\item [\rm (2)]
$(\mathbb{Q}^{n}, \mathcal{O}_{\mathbb{Q}^{n}}(1))$.
\item [\rm (3)]
A scroll over a smooth projective curve.
\item [\rm (4)]
A Del Pezzo manifold.
\item [\rm (5)]
A quadric fibration over a smooth curve.
\item [\rm (6)]
A scroll over a smooth projective surface.
\item [\rm (7)]
Let $(M, \mathcal{A})$ be a reduction of $(X,\mathcal{L})$.
\begin{itemize}
\item [{\rm (7.1)}] $n=4$, $(M, \mathcal{A})=(\mathbb{P}^{4}, \mathcal{O}_{\mathbb{P}^{4}}(2))$.
\item [{\rm (7.2)}]
$n=3$, $(M, \mathcal{A})=(\mathbb{Q}^{3}, \mathcal{O}_{\mathbb{Q}^{3}}(2))$.
\item [{\rm (7.3)}]
$n=3$, $(M, \mathcal{A})=(\mathbb{P}^{3}, \mathcal{O}_{\mathbb{P}^{3}}(3))$.
\item [{\rm (7.4)}]
$n=3$, $M$ is a $\mathbb{P}^{2}$-bundle over a smooth curve $C$ and 
for any fiber $F^{\prime}$ of it, $(F^{\prime}, \mathcal{A}|_{F^{\prime}})\cong(\mathbb{P}^{2}, \mathcal{O}_{\mathbb{P}^{2}}(2))$.
\item [\rm (7.5)]
$K_{M}\sim -(n-2)\mathcal{A}$, that is, 
$(M,\mathcal{A})$ is a Mukai manifold.
\item [\rm (7.6)]
$(M, \mathcal{A})$ is a Del Pezzo fibration over a smooth curve.
\item [\rm (7.7)]
$(M, \mathcal{A})$ is a quadric fibration over a normal surface.
\item [\rm (7.8)]
$n\geq 4$ and $(M,\mathcal{A})$ is a scroll over a normal projective variety of dimension $3$.
\item [\rm (7.9)] $K_{M}+(n-2)\mathcal{A}$ is nef and big.
\end{itemize}
\end{itemize}
\end{Theorem}
\begin{proof}
See \cite[Proposition 7.2.2, Theorem 7.2.4, Theorem 7.3.2, Theorem 7.3.4, 
and Theorem 7.5.3]{BeSo-Book}. 
See also \cite[Chapter II, (11.2), (11.7), and (11.8)]{Fujita-Book}. 
\end{proof}

\begin{Remark}\label{RM2}
Let $(X,\mathcal{L})$ be a polarized manifold with $\dim X=n\geq 3$.
\begin{itemize}
\item [(1)]
$\kappa(K_{X}+(n-2)\mathcal{L})=-\infty$ if and only if $(X,\mathcal{L})$ is one of the types from (1) to (7.4) in Theorem \ref{TH1}.
\item [(2)]
$\kappa(K_{X}+(n-2)\mathcal{L})=0$ if and only if $(X,\mathcal{L})$ is (7.5) in Theorem \ref{TH1}.
\item [(3)]
$\kappa(K_{X}+(n-2)\mathcal{L})\geq 1$ if and only if $(X,\mathcal{L})$ is one of the types from (7.6) to (7.9) in Theorem \ref{TH1}.
\end{itemize}
\end{Remark}

\begin{Definition}\label{DF5}
Let $(X,\mathcal{L})$ be a polarized manifold of dimension $n\geq 3$, and let $(M,\mathcal{A})$ be a reduction of $(X,\mathcal{L})$.
Assume that $K_{M}+(n-2)\mathcal{A}$ is nef and big.
Then for large $m\gg 0$ the morphism $\varphi: M\to W$ associated to $|m(K_{M}+(n-2)\mathcal{A})|$ has connected fibers and normal image $W$.
Then we note that there exists an ample line bundle $\mathcal{K}$ on $W$ such that $K_{M}+(n-2)\mathcal{A}=\varphi^{*}(\mathcal{K})$.
Let $\mathcal{D}:=(\varphi_{*}\mathcal{A})^{\vee\vee}$, where $^{\vee\vee}$ denotes the double dual.
Then the pair $(W,\mathcal{D})$ together with $\varphi$ is called the {\it second reduction of $(X,\mathcal{L})$}. 
\end{Definition}

\begin{Remark}\label{RM6}
(1) If $K_{M}+(n-2)\mathcal{A}$ is nef and big but not ample, then $\varphi$ is equal to the nef value morphism of $\mathcal{A}$.
\\
(2) If $K_{M}+(n-2)\mathcal{A}$ is ample, then $\varphi$ is an isomorphism.
\\
(3) If $n\geq 4$, then $W$ has isolated terminal singularities and is $2$-factorial.
Moreover if $n$ is even, then $X$ is Gorenstein (see \cite[Proposition 7.5.6]{BeSo-Book}).
\end{Remark}

Here we consider a characterization of $(X,\mathcal{L})$ with $\kappa(K_{X}+(n-3)\mathcal{L})=-\infty$.
We note that $\kappa(K_{X}+(n-1)\mathcal{L})=-\infty$ (resp. $\kappa(K_{X}+(n-2)\mathcal{L})=-\infty$) if and only if $(X,\mathcal{L})$ is one of the types from (1) to (3) (resp. from (1) to (7.4)).
Here we consider the case where $\kappa(K_{X}+(n-3)\mathcal{L})=-\infty$.
If $(X,\mathcal{L})$ is one of the types from (1) to (7.8), then $\kappa(K_{X}+(n-3)\mathcal{L})=-\infty$ holds. 
So we assume that $K_{M}+(n-2)\mathcal{A}$ is nef and big.
Then there exist a normal projective variety $W$ with only 2-factorial isolated terminal singularities, a birational morphism $\phi_{2}: M\to W$ and an ample line bundle $\mathcal{K}$ on $W$ such that $K_{M}+(n-2)\mathcal{A}=(\phi_{2})^{*}(\mathcal{K})$.
Let $\mathcal{D}:=(\phi_{2})_{*}(\mathcal{A})^{\vee\vee}$.
Then $\mathcal{D}$ is a 2-Cartier divisor on $W$ and $\mathcal{K}=K_{W}+(n-2)\mathcal{D}$ (see \cite[Lemma 7.5.8]{BeSo-Book}).
Then the pair $(W,\mathcal{D})$ is the second reduction of $(X,\mathcal{L})$
(see Definition \ref{DF5}).
Here we remark that if $K_{M}+(n-2)\mathcal{A}$ is ample, then $(W,\mathcal{K})\cong (M, K_{M}+(n-2)\mathcal{A})$.
\par
Then the following properties hold:
\begin{itemize}
\item [(1)] $\kappa(K_{X}+(n-3)\mathcal{L})=\kappa(K_{W}+(n-3)\mathcal{K})$ holds \cite[Corollary 7.6.2]{BeSo-Book}.
\item [(2)] $(n-2)(K_{W}+(n-3)\mathcal{D})=K_{W}+(n-3)\mathcal{K}$ and $K_{M}+(n-3)\mathcal{A}=\phi_{2}^{*}(K_{W}+(n-3)\mathcal{D})+\Delta$ for an exceptional effective $\mathbb{Q}$-Cartier divisor $\Delta$ of $\phi_{2}$.
Therefore
\begin{eqnarray*}
m(n-2)(K_{X}+(n-3)\mathcal{L})
&=&m(n-2)\phi_{1}^{*}(K_{M}+(n-3)\mathcal{A})+E_{1} \\
&=&m(n-2)\phi_{1}^{*}(\phi_{2}^{*}(K_{W}+(n-3)\mathcal{D}))+E_{1}+m(n-2)\phi_{1}^{*}\Delta\\
&=&m\phi_{1}^{*}\circ \phi_{2}^{*}(K_{W}+(n-3)\mathcal{K})+E_{1}+m(n-2)\phi_{1}^{*}\Delta.
\end{eqnarray*}
(Here $\phi_{1}:X\to M$ is a reduction of $(X,\mathcal{L})$ and $E_{1}$ is a $\phi_{1}$-exceptional effective divisor.)
\item [(3)] $h^{0}((n-2)m(K_{X}+(n-3)\mathcal{L}))=h^{0}(m(K_{W}+(n-3)\mathcal{K}))$ for every integer $m$ with $m\geq 1$.
\end{itemize}

Moreover if $n\geq 4$, then there exists a normal factorial projective variety $M^{\sharp}$ with only isolated terminal singularities and birational morphisms $\phi_{2}^{\sharp}: M\to M^{\sharp}$ and $\psi: M^{\sharp}\to W$ such that $\phi_{2}=\psi\circ\phi_{2}^{\sharp}$.
Then $M^{\sharp}$ is called the {\it factorial stage} (see \cite[7.5.7 Definition-Notation]{BeSo-Book} or \cite[(2.4) Theorem]{Fujita92}).
\par
Here we consider a classification of $(X,\mathcal{L})$ with $\kappa(K_{X}+(n-3)\mathcal{L})=-\infty$ and $n\geq 4$.
First we note that if $(X,\mathcal{L})$ is one of the types from (1) to (7.8), then we see that $\kappa(K_{X}+(n-3)\mathcal{L})=-\infty$.
So we may assume that $K_{M}+(n-2)\mathcal{A}$ is nef and big.
Then there exists the second reduction $(W,\mathcal{D})$ of $X$.
Here we use notation in Definition \ref{DF5}.
If $\tau(\mathcal{K})\leq n-3$, then by above we see that $\kappa(K_{X}+(n-3)\mathcal{L})\geq 0$. (Here $\tau(\mathcal{K})$ denotes the nef value of $\mathcal{K}$.)
So we may assume that $\tau(\mathcal{K})>n-3$.
\par
Here we consider the case of $n=4$.
In this case $M^{\sharp}$ and $W$ are Gorenstein (see \cite[Proposition 7.5.6 and 7.5.7 Definition-Notation]{BeSo-Book}).
Then by the proof of \cite[Section 4]{Fujita92} we see that $(W,\mathcal{K})$ or $M^{\sharp}$ is one of the types in \cite[(4.$\infty$)]{Fujita92}.
If $(W,\mathcal{K})$ or $M^{\sharp}$ is either (4.2), (4.4.0), (4.4.1), (4.4.2), (4.6.0.0), (4.6.0.1.0), (4.6.0.2.1), (4.6.1), (4.7) or (4.8.0) in \cite[(4.$\infty$)]{Fujita92}, then we see that $\kappa(K_{X}+\mathcal{L})=-\infty$.\par
Assume that $(W,\mathcal{K})$ is the type (4.4.4) in \cite[(4.$\infty$)]{Fujita92}. Then we note that $\tau(\mathcal{K})=3$ and there exist a normal Gorenstein projective variety $W_{2}$, an ample line bundle $\mathcal{K}_{2}$ on $W_{2}$ and a birational morphism $\mu: W\to W_{2}$ such that $\mu$ is the simultaneous contraction to distinct smooth points of divisors $E_{i}\cong \mathbb{P}^{3}$ such that $E_{i}\subset \mbox{reg}(W)$, $E_{i}|_{E_{i}}\cong\mathcal{O}_{\mathbb{P}^{3}}(-1)$, $K_{W}+3\mathcal{K}=\mu^{*}(K_{W_{2}}+3\mathcal{K}_{2})$ and $K_{W_{2}}+3\mathcal{K}_{2}$ is ample, that is, $\tau(\mathcal{K}_{2})<3$.
Moreover we infer that $W_{2}$ has the same singularities as $W$ by above.
Since $E_{i}\subset \mbox{reg}(W)$, we have $\psi^{-1}(E_{i})\cong E_{i}$ by the definition of $\psi$.
Hence there exist a normal Gorenstein projective variety $W_{2}^{\sharp}$ and birational morphisms $\mu^{\sharp}: M^{\sharp}\to W_{2}^{\sharp}$ and $\psi^{\sharp}:W_{2}^{\sharp}\to W_{2}$ such that $\mu\circ\psi=\psi^{\sharp}\circ\mu^{\sharp}$.
We note that $\mu^{\sharp}: M^{\sharp}\to W_{2}^{\sharp}$ is the contraction of $\psi^{-1}(E_{i})$ and $W_{2}^{\sharp}$ has the same singularities as $M^{\sharp}$.
$(W_{2},\mathcal{K}_{2})$ is a reduction of $(W,\mathcal{K})$ and is called the {\it $2\frac{1}{2}$ reduction of $(W,\mathcal{K})$} in \cite[(2.2) Theorem-Definition]{BeSo93}.
We also note that $h^{j}(\mathcal{O}_{X})=h^{j}(\mathcal{O}_{M})=h^{j}(\mathcal{O}_{W})=h^{j}(\mathcal{O}_{W_{2}})=h^{j}(\mathcal{O}_{M^{\sharp}})=h^{j}(\mathcal{O}_{W^{\sharp}})$.
For this $\psi^{\sharp}:W_{2}^{\sharp}\to W_{2}$ and $(W_{2},\mathcal{K}_{2})$, we can apply the same argument as in \cite[Section 4]{Fujita92}.
If $\tau(\mathcal{K}_{2})\leq 1$, then we can prove that $\kappa(K_{X}+\mathcal{L})\geq 0$.
So we assume that $\tau(\mathcal{K}_{2})>1$.
Then $(W_{2},\mathcal{K}_{2})$ is either (4.6.0.0), (4.6.0.1.0), (4.6.0.2.1), (4.6.1), (4.6.4), (4.7) or (4.8.0) in \cite[(4.$\infty$)]{Fujita92}.
\par
If $(W_{2},\mathcal{K}_{2})$ is either (4.6.0.0), (4.6.0.1.0), (4.6.0.2.1), (4.6.1), (4.7) or (4.8.0) in \cite[(4.$\infty$)]{Fujita92}, then we see that $\kappa(K_{X}+\mathcal{L})=-\infty$.
\par
If $(W_{2},\mathcal{K}_{2})$ is the type (4.6.4) in \cite[(4.$\infty$)]{Fujita92}, then by the same argument as in \cite[Section 4]{Fujita92} we see that there exist a normal Gorenstein projective variety $W_{3}$, an ample line bundle $\mathcal{K}_{3}$ on $W_{3}$ and a birational morphism $\mu_{2}: W_{2}\to W_{3}$ such that $W_{3}$ has the same singularities as $W_{2}$, $K_{W_{2}}+2\mathcal{K}_{2}=\mu_{2}^{*}(K_{W_{3}}+2\mathcal{K}_{3})$ and 
$K_{W_{3}}+2\mathcal{K}_{3}$ is ample, that is, $\tau(\mathcal{K}_{3})<2$.
Here we note that $\kappa(K_{X}+\mathcal{L})=\kappa(K_{W_{2}}+\mathcal{K}_{2})=\kappa(K_{W_{3}}+\mathcal{K}_{3})$.
\par
If $\tau(\mathcal{K}_{3})\leq 1$, 
then $\kappa(K_{X}+\mathcal{L})=\kappa(K_{W_{3}}+\mathcal{K}_{3})\geq 0$.
\par
If $\tau(\mathcal{K}_{3})>1$, then $(W_{3},\mathcal{K}_{3})$ is either (4.7) or (4.8.0) in \cite[(4.$\infty$)]{Fujita92}by the same argument as in \cite[Section 4]{Fujita92} and we have $\kappa(K_{X}+\mathcal{L})=\kappa(K_{W_{3}}+\mathcal{K}_{3})=-\infty$.
\\
\par
By the above argument, we get the following:
\begin{Theorem}\label{TH2}
Let $(X,\mathcal{L})$ be a polarized manifold of dimension $n=4$.
\\
\\
{\rm (1)} $\kappa(K_{X}+\mathcal{L})\geq 0$ if and only if there exist a normal Gorenstein projective variety $W_{3}$ with only isolated terminal singularities, an ample line bundle $\mathcal{H}_{3}$ on $W_{3}$, and a birational morphism $\Phi: X\to W_{3}$ such that $\tau(\mathcal{H}_{3})\leq 1$ and $h^{0}(2m(K_{X}+\mathcal{L}))=h^{0}(m(K_{W_{3}}+\mathcal{H}_{3}))$ for every positive integer $m$.
\\
\\
{\rm (2)} $\kappa(K_{X}+\mathcal{L})=-\infty$ if and only if $(X,\mathcal{L})$ satisfies one of the following:
\begin{itemize}
\item [\rm (2.1)] $(X,\mathcal{L})$ is either {\rm (1), (2), (3), (4), (5), (6), (7.1), (7.5), (7.6), (7.7)} or {\rm (7.8)} in Theorem {\rm \ref{TH1}}.
\item [\rm (2.2)] There exist a normal projective variety $W_{3}$ with only isolated terminal singularities, an ample line bundle $\mathcal{H}_{3}$ on $W_{3}$, and a birational morphism $\Phi: X\to W_{3}$ such that $(W_{3}, \mathcal{H}_{3})$ is either {\rm (4.2), (4.4.0), (4.4.1), (4.4.2), (4.6.0.0), (4.6.0.1.0), (4.6.0.2.1), (4.6.1), (4.7)} or {\rm (4.8.0)} in {\rm \cite[(4.$\infty$)]{Fujita92}}.
\end{itemize}
\end{Theorem}

By Theorems \ref{Theorem4.1}, \ref{Theorem4.1-big}, \ref{T-TH1}, \ref{Theorem4.2} and \ref{TH2}, we get Theorem \ref{MainTheorem1} in the Introduction.
As a corollary, we have the following.

\begin{Corollary}\label{Corollary4.2}
$m_{4}({\rm SM})\leq 6$.
Here the assumption $({\rm SM})$ is the following.
\\
$({\rm SM})$: $X$ is smooth.
\end{Corollary}

\begin{flushright}
Yoshiaki Fukuma \\
Department of Mathematics \\
Faculty of Science \\
Kochi University \\
Akebono-cho, Kochi 780-8520 \\
Japan \\
E-mail: fukuma@kochi-u.ac.jp
\end{flushright}


\begin{thebibliography}{99}
\bibitem{Ambro}
F. Ambro,
\newblock{\em Ladders on Fano varieties,}
\newblock Algebraic geometry,\ 9,\ J. Math. Sci.,\ 94\ (1999),\ 1126--1135.

\bibitem{Arakawa10}
T. Arakawa,
\newblock{\em Effective nonvanishing of pluri adjoint linear systems,}
\newblock preprint.

\bibitem{BeSo93}
M. C. Beltrametti and A. J. Sommese,
\newblock{\em Special results in adjunction theory in dimension four and five,}
\newblock Ark. Mat.\ 31\ (1993),\ 197--208.

\bibitem{BeSo-Book}
M. C. Beltrametti and A. J. Sommese,
\newblock{\em The adjunction theory of complex projective varieties,}
\newblock de Gruyter Expositions in Math.\ 16\ Walter de Gruyter, 
Berlin, New York,\ (1995).

\bibitem{BeSoWi}
M. C. Beltrametti, A. J. Sommese and J. A. Wi\'sniewski,
\newblock{\em Results on varieties with many lines and their applications to adjunction theory 
{\rm (}with an appendix by M. C. Beltrametti and A. J. Sommese{\rm )},}
\newblock in Complex Algebraic Varieties,\ Bayreuth\ {\rm 1990},\ ed. by K. Hulek, T. Peternell, M. Schneider, and F.-O. Schreyer,\ Lecture Notes in Math.,\ 1507\ (1992),\ 16-38,\ Springer-Verlag,\ New York.

\bibitem{Broustet09}
A. Broustet,
\newblock{\em Non-annulation effective et positivite locale des fibres en droites amples adjoints,} 
\newblock Math. Ann. 343\ (2009),\ 727-755.

\bibitem{BrHo09}
A. Broustet and A. H\"oring,
\newblock{\em Effective non-vanishing conjectures for projective threefolds,} 
\newblock arXiv:0811.3059, to appear in Adv. Geom.

\bibitem{1.5}
J. A. Chen and C. D. Hacon,
\newblock{\em Linear series of irregular varieties,}
\newblock Algebraic Geometry in East Asia (Kyoto 2001),\ 143--153,\ World Sci. Publishing,\ River Edge,\ NJ,\ 2002.

\bibitem{Fujita89}
T. Fujita,
\newblock{\em Remarks on quasi-polarized varieties,}
\newblock Nagoya Math. J.\ 115\ (1989),\ 105--123.

\bibitem{Fujita-Book}
T. Fujita,
\newblock{\em Classification Theories of Polarized Varieties,}
\newblock London Math. Soc. Lecture Note Series\ 155,\ (1990).

\bibitem{Fujita92}
T. Fujita,
\newblock{\em On Kodaira energy and adjoint reduction of polarized manifolds,}
\newblock Manuscripta Math.\ 76\ (1992),\ 59--84.

\bibitem{Fukuma03}
Y. Fukuma,
\newblock{\em On the $c_{r}$-sectional geometric genus of generalized polarized manifolds,}
\newblock Japan. J. Math.\ 29\ (2003),\ 335--355.

\bibitem{Fukuma04}
Y. Fukuma,
\newblock{\em On the sectional geometric genus 
of quasi-polarized varieties, I,}
\newblock Comm. Alg.\ 32\ (2004),\ 1069--1100.

\bibitem{Fukuma04-4}
Y. Fukuma,
\newblock{\em A formula for the sectional geometric genus of quasi-polarized manifolds by using intersection numbers,}
\newblock J. Pure Appl. Algebra\ 194\ (2004),\ 113--126.

\bibitem{Fukuma07}
Y. Fukuma,
\newblock{\em On the dimension of global sections of adjoint bundles for polarized $3$-folds and $4$-folds,}
\newblock J. Pure Appl. Algebra\ 211\ (2007),\ 609--621.

\bibitem{Fukuma08}
Y. Fukuma,
\newblock{\em Invariants of ample line bundles on projective varieties and their applications, I,}
\newblock Kodai Math. J.\ 31\ (2008),\ 219--256.

\bibitem{Fukuma08-3}
Y. Fukuma,
\newblock{\em On the sectional geometric genus of 
multi-polarized manifolds and its application,}
\newblock RIMS Kokyuroku Bessatsu,\ B9\ (2008),\ 97--113.

\bibitem{Fukuma10}
Y. Fukuma,
\newblock{\em Effective non-vanishing of global sections of multiple adjoint bundles for polarized $3$-folds,}
\newblock J. Pure Appl. Algebra\ 215\ (2011),\ 168--184.

\bibitem{Fukuma10-2}
Y. Fukuma,
\newblock{\em Remarks on the second sectional geometric genus of quasi-polarized manifolds and their applications,}
\newblock arXiv:1003.5736.

\bibitem{Hartshorne}
R. Hartshorne,
\newblock{\em Algebraic Geometry,}
\newblock Graduate Texts in Math.\ 52,\ Springer-Verlag,\ New York,\ (1977).

\bibitem{Horing09}
A. H\"oring,
\newblock{\em On a conjecture of Beltrametti and Sommese,}
\newblock arXiv:0912.1295, to appear in J. Algebraic Geom.

\bibitem{Kawamata}
Y. Kawamata,
\newblock{\em On effective non-vanishing and base-point-freeness,}
\newblock Asian J. Math.\ 4\ (2000),\ 173--182.

\bibitem{Kleiman66}
S. L. Kleiman,
\newblock{\em Toward a numerical theory of ampleness,}
\newblock Ann. of Math.\ 84\ (1966)\ 293--344.

\bibitem{Kollar95}
J. Koll\'ar,
\newblock{\em Shafarevich Maps and Automorphic Forms,}
\newblock  M. B. Porter Lectures. Princeton University Press, Princeton,\ 1995.

\bibitem{LPS93}
A. Lanteri, M. Palleschi and D. C. Struppa (Eds.),
\newblock{\em Geometry of complex projective varieties,}
\newblock Proceedings of the conference held in Cetraro, May 28--June 2, 1990.
Seminars and Conferences, 9. Mediterranean Press, Rende, 1993.

\bibitem{Shokurov86}
V. V. Shokurov,
\newblock{\em Theorems on non-vanishing,}
\newblock Math. USSR Izv.\ 26\ (1986),\ 591--604.

\bibitem{Sommese86}
A. J. Sommese,
\newblock{\em On the adjunction theoretic structure of projective varieties,}
\newblock Complex analysis and algebraic geometry (G\"ottingen, 1985),\ 175--213,
\ Lecture Notes in Math.,\ 1194,\ Springer,\ Berlin,\ 1986.

\end{thebibliography}
\end{document}